\documentclass{article}
\usepackage[T1,T2A]{fontenc}
\usepackage[utf8]{inputenc}
\usepackage[english]{babel}
\usepackage{amsmath,amsfonts, amssymb, amsthm,diffcoeff,geometry}
\usepackage{thmtools}
\usepackage[final, 
colorlinks=true,
linkcolor=blue,
citecolor = magenta]{hyperref}
\usepackage{cleveref}
\usepackage[backend=biber,language =auto, autolang =other]{biblatex}
\renewcommand{\epsilon}{\varepsilon}
\newcommand{\R}{\mathbb{R}}
\newcommand{\N}{\mathbb{N}}
\newcommand{\m}{\mathfrak{m}}
\newcommand{\narrowto}{\rightharpoonup}
\DeclareMathOperator{\spt}{spt}
\DeclareMathOperator{\age}{age}
\DeclareMathOperator{\diam}{diam}

\newcommand{\Ical}{\mathcal{I}}
\newcommand{\Pcal}{\mathcal{P}}

\DeclareMathOperator{\id}{id}

\DeclareMathOperator{\CD}{\mathsf{CD}}

\DeclareMathOperator{\TCD}{\mathsf{TCD}}
\DeclareMathOperator{\TBM}{\mathsf{TBM}}
\DeclareMathOperator{\sTBM}{\mathsf{sTBM}}
\DeclareMathOperator{\TMCP}{\mathsf{TMCP}}

\DeclareMathOperator{\Opt}{Opt}
\DeclareMathOperator{\OptGeo}{OptGeo}
\DeclareMathOperator{\TOpt}{TOpt}
\DeclareMathOperator{\TOptGeo}{TOptGeo}

\DeclareMathOperator{\Geo}{Geo}

\DeclareMathOperator{\Ent}{Ent}

\newtheorem{theorem}{Theorem}[section]
\newtheorem{lemma}[theorem]{Lemma}
\newtheorem{prop}[theorem]{Proposition}
\newtheorem{corollary}[theorem]{Corollary}

\theoremstyle{definition}
\newtheorem{defn}{Definition}[section]

\title{The equivalence between timelike Ricci curvature and the timelike Brunn-Minkowski inequality on synthetic Lorentzian spaces} 

\author{Osama Farooqui}

\begin{document}
\maketitle
\tableofcontents

\abstract{We introduce the strong $q$-timelike Brunn-Minkowski condition $\mathsf{sTBM}_q(K,N)$ on synthetic Lorentzian spaces, for $0<q<1$. We show that, in the timelike $q$-essentially non-branching setting, the $q$-timelike curvature dimension condition $\mathsf{TCD}_q(K,N)$ is equivalent to $\mathsf{TBM}_q(K,N^+)$, and that the entropic $q$-timelike curvature dimension condition $\mathsf{TCD}_q^e(K,N)$ is equivalent to the reduced $\mathsf{sTBM}$ condition, $\mathsf{sTBM}_q^*(K,N)$. This extends, to a non-smooth setting, our earlier work in proving the equivalence between Ricci curvature and the Brunn-Minkowski inequality on $C^2$ spacetimes.}
\section{Introduction}

This report is concerned with the Brunn-Minkowski inequality on synthetic Lorentzian spaces. Lorentzian spaces carry a causal structure induced by the so-called ``time separation function'', in analogy with the metric structure induced by the distance function on complete Riemannian manifolds. Continuing the analogy, one might be interested in forgetting the smooth structure of the manifold and studying only the causal structure. Work in this direction is old indeed and has been revisited numerous times in the last 100 years, see \cite{Robb, KP, Noldus, KS}. The program is less intuitive than it's Riemannian counterpart, in part due to the lack of topological data that can be extracted from the time separation function, and moreover due to recent interest in constructing spaces which are amenable to the tools of optimal transport. The first of these more modern definitions of ``synthetic'' Lorentzian spaces is Kunzinger--S\"amann's {\it Lorentzian length space} \cite{KS}. Soon after, McCann introduced the notion of {\it metric spacetime} \cite{McCannII}; he later helped refine this definition as part of the octet of Beran {\it et al.\@} \cite{Octet}. Within these two synthetic Lorentzian spaces, one can define the $\TCD$ condition -- a non-smooth analogue of uniform Ricci lower bounds in timelike directions on a Lorentzian manifold. The $\TCD$ condition was first introduced in the smooth setting by McCann \cite{McCann}, and independently, Mondino and Suhr \cite{MS}. Later Braun \cite{Braun} introduced the $\TCD_q$ condition on Lorentzian length spaces, following Cavalletti and Mondino's \cite{CM} introduction of the ``entropic'' $\TCD^e_q$ condition. Beran {\it et al.\@} introduced a weaker condition, $\TMCP_h^+$, on metric spacetimes, but it's fairly easy to see what the correct notion of $\TCD$ would be on these spaces. 

As we mentioned earlier, Lorentzian synthetic spaces are analogues of metric measure spaces. Similarly is $\TCD$ an analogue of the $\CD$ condition, introduced by Lott--Villani \cite{LV}, and independently, Sturm \cite{SturmI,SturmII}. Roughly speaking, $\CD$ is a convexity requirement of the entropy functional along a geodesic of probability measures. One of the strengths of the $\CD$ formulation is the recovery of well-known geometric inequalities, notably the Brunn-Minkowski inequality, suitably modified to account for the lower Ricci bounds. Recall that the Brunn-Minkowski inequality is a concavity requirement of the $n$-th root of the volume of compact sets along a geodesic interpolation. 

Recently Magnabosco, Portinale, and Rossi \cite{MPRI} proved that the Brunn-Minkowski inequality is sufficient to deduce the convexity of the entropy on weighted Riemannian manifolds. In \cite{MPRII} the same authors introduced the notion of the {\it strong Brunn-Minkowski inequality}, which requires concavity of the $n$-th root of the volume along a geodesic of probability measures, and showed that the condition is equivalent to $\CD(K,N)$ in the essentially non-branching metric measure space setting. 

In a previous paper \cite{smoothTBMtoTCD}, we proved an equivalence between the timelike Brunn Minkowski inequality and timelike Ricci lower bounds on $C^2$-smooth weighted spacetimes. In this paper we introduce a {\it strong timelike Brunn-Minkowski inequality}, and reproduce the results of \cite{MPRII} in the synthetic Lorentzian setting. Our main results are \Cref{sTBMtoTCDe} and \Cref{sTBMtoTCD}, which can be summarized as follows. \\

\noindent {\bf Theorem}: {\it Fix $0<q<1$. Let $X$ be a timelike $q$-essentially non-branching globally hyperbolic synthetic spacetime with $\ell_+$ continuous. Then $X\in \sTBM_q^*(K,N)$ iff $X\in \TCD_q^e(K,N)$, and $X\in \sTBM_q(K,N^+)$ iff $X\in \TCD_q(K,N)$.}\\

In section 2 we recall the definition of metric spacetimes, optimal transport on these spaces, and the curvature conditions $\TCD_q^{(e)}(K,N)$, and $\sTBM_q^{(*)}(K,N)$. We also discuss the notion of ``simple measures'', which are central to the proof, as they were in \cite{MPRII}; in that paper they are referred to as ``step measures'', whereas we opt for the term ``simple measure'' in analogy with simple functions. In section 3 we prove our main results. We begin by verifying that $\TCD$ spaces indeed recover the $\sTBM$ property. We then show that $\sTBM$ spaces satisfy a weaker curvature condition, $\TMCP$, which simplifies our analysis significantly and allows us to use the arguments in \cite{MPRII} to prove our main results. 
\section{Preliminaries} 
\subsection{Lorentzian Synthetic Spaces}
Let us reiterate that there is some freedom in choosing the spaces that we will work in. Our framework will be based on Beran {\it et al.\@}'s metric spacetimes. We want to include some additional properties within the construction of our spaces, and want as well an efficient name to refer to spaces with such properties; we opt for the somewhat generic label of ``synthetic spacetime''. We will endeavour to draw a connection between our synthetic spacetimes and the other examples highlighted in the introduction, whenever possible. Most of, if not all, of the following definitions are drawn from \cite{McCannII}. 
\subsubsection{Causal Spaces}
\begin{defn}
A {\it causal space} is a pair $(X,\ell)$ consisting of a set $X$, and a {\it time separation function} $\ell: X\times X\to \{-\infty\}\cup [0,\infty]$, with the following properties
\begin{enumerate}
\item {\it (Causal reflexivity):} $\ell(x,x)\geq 0\quad \forall x\in X$;
\item {\it (Reverse triangle inequality):} $\ell(x,y)\geq \ell(x,z)+\ell(z,y) \quad \forall x,y,z\in X$,
\end{enumerate}
with the convention that $-\infty+a=a+(-\infty) =-\infty$ for all $a\in \{-\infty\}\cup [0,\infty]$. We often use the notation $\ell_+:= \max\{\ell, 0\}$. 
\end{defn}
The time separation induces a binary relation, called the {\it causal order}. We say $x$ is in the {\it causal past} of $y$, and write $x\leq y$, iff $\ell(x,y)\geq 0$. Similarly, we say $x$ is in the {\it chronological past} of $y$, and write $x\ll y$ iff $\ell(x,y)>0$.  

We remark that if the relation $\leq$ is antisymmetric, i.e $x\leq y$ and $y\leq x$ imply $x=y$, then $\leq$ is a partial order, and $(X,\leq, \ll)$ is a causal set in the sense of Kronheimer--Penrose \cite{KP}. Antisymmetry does not hold in our general setting, but is inherited under stronger assumptions on $(X,\ell)$. 

We define the {\it causal set} 
\[X_\leq^2:=\{(x,y)\in X^2: x\leq y\},\]
 and the {\it chronological set} 
 \[X_\ll^2:= \{(x,y)\in X^2: x\ll y\}.\] 

For each $x\in M$, we denote the {\it causal past} of $x$ by $J^-(x):=\{y\in X: y\leq x\}$, and the {\it causal future} of $x$ by $J^+(x):=\{y\in X: x\leq y\}$. Then for a set $A\subset X$, we define the causal future/past of $A$ as
\[J^\pm(A):= \bigcup_{a\in A} J^\pm(a).\]
For two sets $A,B\subset X$, we define the {\it causal emerald} of $A$ and $B$, $J(A,B)$, as 
\[J(A,B) := J^+(A)\cap J^-(B).\]

One could also define similar constructions for the chronological relation, with similar notation; commonly, one replaces $J$ with an $I$. Such sets will not be relevant to our discussion here. 

\begin{defn}[Paths and the age of a path]
A {\it path} is a map $\gamma:[0,1]\to X$. A path is {\it causal} if it is increasing in the causal order, i.e $s\leq t\implies \gamma_s\leq \gamma_t$. A causal path is {\it timelike} $\ell(\gamma_s,\gamma_t)>0$ for all $s\leq t$, and {\it null} if $\ell(\gamma_s,\gamma_t)=0$ for all $s\leq t$. The {\it age} of a causal path is defined by 
\[\age(\gamma):= \inf\left\{\sum_{i=1}^N\ell(\gamma_{t_{i-1}},\gamma_{t_i}): N\in \N, 0=t_0\leq t_1\leq\dots\leq t_N=1\right\}.\]
\end{defn}
\begin{prop}[Properties of $\age$]
Let $(X,\ell)$ be a causal space. Then $\age$ is additive under concatenation, and invariant under (continuous, strictly increasing) reparametrization. 
\end{prop}
\begin{proof}
Let $\gamma$ and $\sigma$ be two causal paths such that $\gamma_1=\sigma_0$. Define the concatenation 
\[\gamma\oplus \sigma := \begin{cases}
\gamma_{2t} & 0\leq t\leq \frac12; \\
\sigma_{2t-1} & \frac12\leq t\leq 1.
\end{cases}\]

Let $\{t_i\}_{i=1}^N$ be a partition of $[0,1]$. Let $t_j$ be the largest element which is smaller than $1/2$. Then $\{t_i\}_{i=1}^j\cup \{1/2\}$ and $\{1/2\}\cup \{t_i\}_{i=j+1}^N$ are partitions of $[0,1/2]$ and $[1/2, 1]$, and 
\begin{align*}
\age(\gamma)+\age(\sigma)\leq& \left(\sum_{i=1}^j \ell(\gamma_{2t_{i-1}}, \gamma_{2t_{i}})+ \ell(\gamma_{2t_{i}},\gamma_{1})\right)+ \left(\ell(\sigma_{0}, \sigma_{2t_{j+1}-1})+\sum_{i=j+1}^N \ell(\sigma_{2t_{i-1}-1}, \sigma_{2t_{i}-1})\right)\\
=& \sum_{i=1}^j \ell((\gamma\oplus \sigma)_{t_{i-1}},(\gamma\oplus \sigma)_{t_{i}})+ \ell((\gamma\oplus \sigma)_{t_{j}},(\gamma\oplus \sigma)_{1/2})\\
&+ \ell((\gamma\oplus \sigma)_{1/2},(\gamma\oplus \sigma)_{t_{j+1}}) + \sum_{i=j+1}^N \ell((\gamma\oplus \sigma)_{t_{i-1}},(\gamma\oplus \sigma)_{t_{i}})\\
&\leq \sum_{i=1}^N \ell((\gamma\oplus \sigma)_{t_{i-1}},(\gamma\oplus \sigma)_{t_{i}}).
\end{align*}
Taking infimum over all partitions $\{t_i\}$, we obtain 
\[\age(\gamma)+\age(\sigma)\leq \age(\gamma\oplus\sigma).\]
Similarly, we may show that 
\[\age(\gamma\oplus\sigma)\leq \age(\gamma)+\age(\sigma),\]
and thus we achieve equality. Now we show invariance under reparametrization. Let $\tilde \gamma$ be a reparametrization of a causal path $\gamma$; that is, $\tilde\gamma =\gamma\circ\phi$, where $\phi:[0,1]\to [0,1]$ is a continuous and strictly increasing bijection.
If $\{t_i\}_{i=0}^N$ is a partition of $[0,1]$, then so too is $\{\phi(t_i)\}_{i=0}^N$, 
\[\age(\gamma) \leq \sum_{i=1}^N\ell(\gamma_{\phi(t_{i-1})},\gamma_{\phi(t_{i})}) = \sum_{i=1}^N \ell(\tilde \gamma_{t_{i-1}},\tilde \gamma_{t_i}).\]
Taking infimums over $\{t_i\}$ gives $\age(\gamma)\leq \age(\tilde \gamma)$. Using the invertibility of $\phi$, we may similarly prove that $\age(\tilde\gamma)\leq \age(\gamma)$. 
\end{proof}

\begin{defn}[Age maximizing and geodesic paths]
 A causal path $\gamma$ is {\it age-maximizing} if $\age(\gamma)=\ell(\gamma_0, \gamma_1)$. A causal path $\gamma$ is a {\it geodesic path} if there is a reparametrization $\tilde \gamma$ of $\gamma$ such that $\ell(\tilde \gamma_s,\tilde \gamma_t)=(t-s)\ell(\tilde \gamma_0,\tilde \gamma_1)$ for all $0\leq s\leq t\leq 1$. 
\end{defn}
 Clearly every geodesic path is age-maximizing. \Cref{agemaxtogeodesic} below will give conditions under which the converse also holds. 
\subsubsection{Topology on Causal Spaces} 

\begin{defn}[Topological and metric causal spaces] A {\it topological causal space} is a triple $(X,\ell,\tau)$ consisting of a causal space $(X,\ell)$ and a topology $\tau$. When this topology is metrizable, we say that $X$ is a {\it metrizable causal space}. A metrizable causal space is {\it Polish} if the topology can be generated by a complete and separable metric. 
\end{defn}

\begin{defn}[Age-maximizing and geodesic curves]
On a topological causal space,  a {\it curve} is a continuous path, i.e.\ a continuous map $\gamma:[0,1]\to X$. A continuous age-maximizing path is called an {\it age-maximizing curve} and a continuous geodesic path is called a {\it geodesic}. 
 \end{defn}
 \begin{prop}[Equivalence between age-maximizing and geodesic curves]\label{agemaxtogeodesic}
Let $(X,\ell ,\tau)$ be a metric spacetime, and suppose $\ell$ is upper semicontinuous and does not attain $\infty$. Then every age-maximizing curve is a geodesic. 
 \end{prop}
\begin{proof}
Let $\gamma$ be an age-maximizing curve. Since $\ell_+$ is finite, $\age(\gamma)=A<\infty$. If $M=0$, then $\gamma$ is null, and thus must be a geodesic. So assume $A\neq 0$. 

Define $\phi:[0,1]\to [0,1]$ via $\phi(t):= A^{-1}\age(\gamma\vert_{[0,t]})$. It is easy to show that $\phi$ is strictly increasing. It is also continuous, since for $s<t$,  $\phi(t)-\phi(s)=A^{-1}\age(\gamma\vert_{[s,t]})\leq A^{-1}\ell(\gamma(s),\gamma(t))$, and the map $(s,t)\mapsto \ell(\gamma(s),\gamma(t))$ is upper semicontinuous and vanishes along the diagonal. Then  $\tilde \gamma = \gamma\circ \phi^{-1}$ is a reparametrization of $\gamma$. Moreover we have 
\[\age(\tilde \gamma\vert_{[0, t]})=\age(\gamma\vert_{[0,\phi^{-1}(t)]})= A\phi(\phi^{-1}(t))=At=t\ell(\gamma_0, \gamma_1),\]
and therefore 
\[\ell(\tilde\gamma_s,\tilde \gamma_t)=\age(\tilde \gamma\vert_{[s,t]})= (t-s)\ell(\gamma_0,\gamma_1).\]
\end{proof}
From now on we assume that every geodesic $\gamma$ is parametrized such that $\ell(\gamma_s,\gamma_t)=(t-s)\ell(\gamma_0,\gamma_1)$ for all $0\leq s\leq t\leq 1$. We now show that timelike geodesic paths can be made to be continuous under certain conditions.

A central condition is {\it global hyperbolicity}, which is somehow analogous to properness in metric spaces. 

\begin{defn}[Global hyperbolicity]
A topological causal space $(X,\ell, \tau)$ is {\it globally hyperbolic} if for any two compact subsets $C_1,C_2\subset X$, the causal emerald $J(C_1,C_2)$ is compact. 
\end{defn}

This definition of global hyperbolicity excludes the standard condition of {\it non-totally imprisoning} which is usually paired with the compactness of diamonds. However, Minguzzi showed that the standard definition is equivalent to compactness of diamonds and causal closedness. In our setting, we inherit causal closedness from upper semicontinuity of the time separation, so we reserve the term ``globally hyperbolic'' for the compactness of emeralds only. This keeps us close to \cite{Octet}.

\begin{lemma}[Continuity of timelike geodesic paths]\label{geodesicpathtocurve}
Let $(X,\ell, \tau)$ be a topological causal space such that the following 4  properties hold: 
\begin{enumerate}
\item Every pair of causally related points can be joined by an age-maximizing curve;
\item Every age-maximizing curve joining a pair of timelike related points is timelike;
\item The time separation $\ell$ is upper semi-continuous;
\item The space $X$ is globally hyperbolic.
\end{enumerate}
Then every timelike geodesic path of finite age is continuous.
\end{lemma}
{\bf Remark:} Any topological causal space satisfying the property {\it 2} is said to be {\it regular}. Notice that if we add the condition that $\ell$ does not attain $\infty$, then by \Cref{agemaxtogeodesic}, every age-maximizing curve is a geodesic, and so regularity is satisfied.

\begin{proof}
Let $\gamma$ be a timelike geodesic path with finite age. Let $s\in(0,1]$. Then $J(\gamma_0, \gamma_s)$ is compact, and $\gamma_t\in J(\gamma_0,\gamma_s)$ for all $t\in (0,s)$. Let $\{t_n\}$ be a sequence which increases to $s$. By compactness, $\gamma_{t_n}$ admits some subsequential limit point $p\in J(\gamma_0, \gamma_s)$. Upper semicontinuity of $\ell$ then confirms that $\ell(\gamma_0, p)\geq \ell(\gamma_0, \gamma_s)$ and $\ell(p, \gamma_s)\geq 0$. 

We may find two age-maximizing curves $\sigma_1$ and $\sigma_2$ joining $\gamma_0$ to $p$, and joining $p$ to $\gamma_s$ respectively. Concatenating these curves gives a curve $\sigma:= \sigma_1\oplus \sigma_2$ such that 
\[\age(\sigma)= \age(\sigma_1)+\age(\sigma_2)= \ell(\gamma_0, p)+\ell(p, \gamma_1)\geq \ell(\gamma_0, \gamma_s).\]

Since of course $\age(\sigma)\leq \ell(\gamma_0, \gamma_s)$, we obtain equality in all above inequalities, which verifies that $\sigma$ is an age maximizing curve joining $\gamma_0$ to $\gamma_1$, and moreover gives that $\ell(\gamma_0, p)=\ell(\gamma_0, \gamma_s)$ and $\ell(p, \gamma_s)= 0$. But this last equality violates regularity unless $p=\gamma_s$. Arbitrariness of the (sub)sequence $\{t_n\}$ then shows that $\displaystyle\lim_{t\to s^-} \gamma_t= \gamma_s$ for all $s\in (0,1]$. A similar proof shows that $\displaystyle\lim_{t\to s^+}\gamma_t= \gamma_s$ for all $s\in [0,1)$. Thus $\gamma$ is continuous. 

\end{proof}

\begin{defn}[Synthetic spacetime]\label{spacetimedef} We say that a topological causal space $X= (X,\ell,\tau)$ is a {\it synthetic spacetime} if the following hold: 
\begin{enumerate}
\item  The topology $\tau$ is Polish;
\item The time separation $\ell$ is upper semicontinuous and does not attain $+\infty$;
\item ({\it Causally geodesic}) Every pair of causally related points can be joined by an causal geodesic curve.
\end{enumerate}
\end{defn}

Throughout this paper, we will be working in synthetic spacetimes which are globally hyperbolic and in which $\ell_+:= \max\{\ell, 0\}$ is continuous. However we would also like to restrict our definition to those properties which are inherited by the space of probability measures (see \Cref{spacetimestructure}), and the two properties noted above are not in this class.

{\bf Remark:} (Comparison with other Lorentzian synthetic spaces). As we said, our definition of spacetimes are close to the `Polish metric spacetimes' of \cite{Octet}.\ For them, the basic working assumptions are Polishness of the topology, antisymmetry of the causal order and so called {\it forward completeness}: every sequence which is increasing and bounded above in the causal order attains a limit. For us, antisymmetry is guaranteed by the fact that $\ell$ does not attain $+\infty$, and completeness (both forward and backward) will be implied by global hyperbolicity. 

We are somewhat further away from the Lorentzian length spaces of Kunzinger and S\"amann. There are some technical assumptions that stand in the way of drawing a general equivalence, but in light of \cite{McCannII}, we can say that our globally hyperbolic synthetic spacetime becomes their globally hyperbolic Lorentzian length space if one first fixes a metric $d$ which generates the topology, and then adds lower semicontinuity of $\ell_+:= \max\{\ell, 0\}$, and non-emptiness of the sets $I^\pm(x)$ for all $x\in X$. One should also note that for them, curves are always $d$-Lipschitz continuous.

\subsection{Optimal Transport on Lorentzian Synthetic Spaces}

\subsubsection{Topology on Probability Measures}
A {\it measured topological causal space} is a quadruple $(X,\ell,\tau, \m)$ consisting of a topological causal space $(X,d,\ell)$ and a Radon measure $\m$ with full support, meaning $\spt\m=X$. Given $(X,\ell, \tau, \m)$, we let $\Pcal(X)$ denote the set of (Borel) probability measures on $X$, and $\Pcal_c^{ac}(X)\subset \Pcal(X)$ the set of probability measures with compact support which are absolutely continuous with respect to $\m$. The space $\Pcal(X)$ is equipped with the topology of {\it narrow convergence}. We say $\mu^n$ narrowly converges to $\mu$, and write $\mu^n\narrowto\mu$, if 
\[\int_X f(x)d\mu^n(x)\to \int f(x)d\mu(x) \quad \forall f\in C_b(X),\]
where $C_b(X)$ denotes the set of continuous bounded functions $f:X\to \R$. A useful characterization of narrow convergence is given by {\it Portmanteau's theorem}\footnote{In \cite[pp. 26-27]{BraunMcCann}, Braun and McCann suggest that 'Portmanteau's Theorem' is a misattribution. Billingsley attributes this result to Portmanteau, but the referenced paper and author allegedly don't exist. On the other hand, a proof does appear in an earlier paper by A.D. Alexandroff \cite[\S 16]{Alexandroff}-- following Bogachev's suggestion \cite[pp 454]{Bogachev}, Braun and McCann call it {\it Alexandroff's Theorem}.}: 
\begin{prop}[Portmanteau]
Let $(X,\ell, \tau, \m)$ be a measured metrizable causal space. Let $\{\mu^n\}_{n=1}^\infty\subset \Pcal(X)$, and $\mu\in X$. Then the following are equivalent: 
\begin{enumerate}
\item $\mu^n\narrowto \mu$;
\item For all  lower semicontinuous bounded functions $f:X\to \R$,  \[\liminf_{n\to\infty}\int fd\mu^n \geq \int fd\mu; \]
\item For all upper semicontinuous bounded functions $f:X\to \R$, 
\[\limsup_{n\to\infty}\int fd\mu^n\leq \int f d\mu.\]
\end{enumerate}
\end{prop}

Narrow compactness is guaranteed by Prokhorov's theorem: 
\begin{defn}[Tightness]
 A set $M\subset \Pcal(X)$ is called {\it tight} if for each $\epsilon>0$, there is a compact set $K\subset X$ such that for all $\mu\in M$, 
\[\mu(X\setminus K)<\epsilon.\]
\end{defn}

\begin{prop}[Prokhorov]
Let $X$ be separable. A set $M\subset \Pcal(X)$ is narrowly precompact iff it is tight. 
\end{prop}

\begin{prop}[Condition for tightness]
Let $X$ be a compact metric space. Then any collection of measures is tight. 
\end{prop}
It is well known that on metric spaces, the narrow topology is metrizable. Given a metrizer $d$ of $X$ and $p\in [1, \infty)$, we define the {\it $p$-Wasserstein metric}{\footnote{Another misattribution. It is well known that Kantorovich and Rubinstein were the first to study this metric, and Wasserstein never worked with it at all.} $W_p:\Pcal(X)^2\to [0,\infty]$ as 
\begin{align}\label{W_Pdefn}
W_p(\mu,\nu):= \inf_{\pi\in\Pi(\mu,\nu)} \left(\int_{X\times X}d^p(x,y)d\pi(x,y)\right)^{1/p} 
\end{align}
 where $\Pi(\mu,\nu):= \{\pi\in \Pcal(X\times X): (p_1)_\#\pi=\mu, (p_2)_\#\pi=\nu\}$ denotes the set of {\it admissible couplings} between $\mu$ and $\nu$. Here $p_1(x,y):=x$ and  $p_2(x,y)=y$ denote the projection maps onto the first and second component, respectively. 
 \begin{prop}[Topology of Wasserstein space, \text{\cite[Prop. 7.1.5]{AGS}}]
 Let $(X,d)$ be a metric space, and $p\in [1,\infty)$. We have the following 
 \begin{enumerate}
 \item If $(X,d)$ is Polish, then $(\Pcal(X),W_p)$ is Polish. 
 \item For any sequence $\{\mu_n\}\subset \Pcal(X)$ and $\mu\in \Pcal(X)$, $\mu_n\xrightarrow{W_p}\mu$ iff $\mu_n\narrowto \mu$ and ${\sup_n\int d^p(x,x_0)d\mu^n<\infty}$ for some, and hence any, $x_0\in X$.

 \end{enumerate}
 In particular, if $X$ is compact, then convergence in $W_p$ is equivalent to narrow convergence. 
 \end{prop}
\subsubsection{The Causal Structure on Probability Measures}
We now define the causal structure on $\Pcal(X)$. For $q\in(0,1)$, we define the {\it $q$-Eckstein-Miller time separation}, $\ell_q$, as 
\begin{align}\label{l_qdefn}
\ell_q(\mu_0,\mu_1):= \sup_{\pi\in \Pi(\mu_0,\mu_1)} \left(\int_{X\times X} \ell^q(x,y)d\pi(x,y)\right)^{1/q},
\end{align}
with the convention $(-\infty)^q=(-\infty)^{1/q}= -\infty$. The name is in reference to \cite[Def.\ 3]{EcksteinMiller}, who to our knowledge were the first to introduce a time separation of this form. 
\begin{lemma}
Let $(X,\ell,\tau)$ be a metrizable causal space. For each $q\in (0, 1)$, $\ell_q$ is a time separation.
\end{lemma}
\begin{proof}
Let $\mu\in \Pcal(X)$. Then $(\id,\id)_\#\mu$ is an admissible plan, and therefore 
\[\ell_q(\mu,\mu)\geq \left(\int \ell^q(x,x)d\mu(x)\right)^{1/q}\geq 0.\]

Now we prove the reverse triangle inequality. Let $\mu, \nu, \upsilon\in \Pcal(X)$. We want to show that 
\[\ell_q(\mu,\nu)\geq \ell_q(\mu, \upsilon)+\ell_q(\upsilon, \nu). \]
WLOG assume that $\ell_q(\mu,\upsilon)\geq 0$ and $\ell_q(\upsilon, \nu)\geq 0$. In particular, for each $\epsilon>0$, there must exist admissible couplings $\pi_1^n\in \Pi(\mu,\upsilon)$, $\pi_2^n\in \Pi(\upsilon, \nu)$ such that 
\[\ell_q(\mu, \upsilon)\leq \left(\int \ell^q(x,y)d\pi_1^n(x,y)\right)^{1/q}+\epsilon; \quad \ell_q(\upsilon, \nu) \leq \left(\int\ell^q(x,y)d\pi_2^n(x,y)\right)^{1/q}+\epsilon.\]

Disintegrate $\pi_1^n$ along $p_2$, and $\pi_2^n$ along $p_1$ to get $d\pi_1^n:=  (\pi_1^n)_x d\eta(x)$ and $d\pi_2^n :=  (\pi_2^n)_x d\eta(x)$, and define 
\[\bar \pi:= \int (\pi_1^n)_x\otimes (\pi_2^n)_x d\eta(x) \in \Pcal(X\times X\times X)\]
We define $\pi:= (p_{1,3})_\# \bar \pi$, where $p_{1,3}(x,y,z)= (x,z)$. One can check by applying the projection maps that $\pi\in \Pi(\mu, \nu)$. From Minkowski's reverse inequality, we then see that 
\begin{align*}
\ell_q(\mu, \nu) &\geq \left(\int \ell^q(x,y) d\pi(x,y)\right)^{1/q} = \left(\int \ell^q(x,z) d\bar \pi(x,y,z)\right)^{1/q}\\
&\geq \left(\int \ell^q(x,y) d\bar \pi(x,y,z)\right)^{1/q}+\left(\int \ell^q(y,z)d\bar \pi(x,y,z)\right)^{1/q}\\
&= \ell_q(\mu, \upsilon)+\ell_q(\upsilon,\nu)-2\epsilon.
\end{align*}
As $\epsilon$ was arbitrary, we achieve the result. 
\end{proof}

 Thus for each $q\in (0,1)$, $(\Pcal(X), \ell_q)$ is a metrizable causal space when equipped with the narrow topology. We will see later that when $X$ is a compact synthetic spacetime, $\Pcal(X)$ is also a compact synthetic spacetime (\Cref{spacetimestructure}). 
 
\begin{defn}[Optimal couplings] Given $\mu,\nu\in \Pcal(X)$, we call $\pi\in \Pi(\mu,\nu)$ a {\it $q$-optimal coupling} if $\pi$ attains the supremum in \eqref{l_qdefn}. The set of $q$-optimal couplings between $\mu$ and $\nu$ is denoted $\Opt_q(\mu, \nu)$. A coupling $\pi$ is {\it timelike} if it is concentrated on $X_\ll^2$, and the set of timelike $q$-optimal couplings is denoted $\TOpt_q(\mu,\nu)$. 
\end{defn}

When $\ell$ is upper semicontinuous, standard optimal transport techniques then give the existence of $q$-optimal couplings under certain conditions outlined below. 

\begin{defn}[Dualizability]
We say a pair of measures $(\mu, \nu)\in\Pcal(X)^2$ are {\it dualizable} if there exists lower semi-continuous functions $a\in L^1(\mu)$ and $b\in L^1(\nu)$ such that $\ell(x,y)\leq a(x)+b(y)$. 
\end{defn}

\begin{prop}[Existence of optimal couplings, \text{\cite[Lems.\ 4.3,4.4]{Villani}}] Let $(X,\ell, \tau, \m)$ be a measured topological causal space with $\ell$ upper-semicontinuous. Then for any dualizable pair $(\mu,\nu)\in \Pcal(X)^2$, $\Opt_q(\mu, \nu)$ is non-empty. 
\end{prop}

A useful characterization of optimal couplings is that of cyclical monotonicity.

\begin{defn}[Cyclical monotonicity] Let $c:X^2\to \{-\infty\}\cup [0,\infty]$. A set $\Gamma\subset X^2$ is {\it $c$-cyclically monotone} if for each $N\in \N$, any $\{(x_i,y_i)\}_{i=1}^N\subset \Gamma$, and any permutation $\sigma\in S_N$ of $N$ elements, we have 
\[\sum_{i=1}^Nc(x_i, y_i)\geq \sum_{i=1}^N c(x_i, y_{\sigma(i)})\]
\end{defn}

\begin{prop}[Optimality and cyclical monotonicity, \text{\cite[Prop.\ 2.8]{CM}}]\label{cyclicallymonotoneoptimal} Let $(\mu,\nu)$ be a dualizable pair of measures, and let $\pi\in \Pi(\mu, \nu)$. If $\pi\in \Opt_q(\mu,\nu)$, then it is concentrated on some $\ell_q$-cyclically monotone set. Conversely, if $\pi$ is concentrated on some $\ell_q$-cyclically monotone set contained in $X_\ll^2$, then $\pi\in\Opt_q(\mu,\nu)$. 
\end{prop}
We now introduce the various notions of chronology available on $\Pcal(X)$. The weakest notion is ``$q$-timelike dualizability'', and the strongest is ``totally timelike dualizability". The term ``totally timelike'' is, to our knowledge, new in the literature. It was chosen at the suggestion of McCann. In \Cref{midpointTCDe} we approximate $q$-timelike pairs of measures with totally timelike ones, and so in fact all the notions in-between are immaterial for our results; nonetheless we include them for completeness.

\begin{defn}[Chronology for measures]\label{chronologynotions} Let $(\mu,\nu)\in\Pcal(X)^2$ be dualizable. Let $q\in (0,1)$. We say that the pair $(\mu,\nu)$ is... 
\begin{enumerate}
\item ... {\it $q$-timelike} if there exists a timelike $q$-optimal coupling $\pi\in \TOpt_q(\mu,\nu)$;
\item ...  {\it strictly $q$-timelike} if every $q$-optimal coupling is timelike; 
\item ... {\it strongly $q$-timelike} if every $q$-optimal coupling is concentrated on the same $\ell_q$-cyclically monotone set $\Gamma\subset X_\ll^2$;
\item ... {\it $q$-separated} if  every $q$-optimal coupling is concentrated on the same {\it compact} $\ell_q$-cyclically monotone set $\Gamma\subset X_\ll^2$; 
\item ... {\it totally timelike} if $\spt\mu\times \spt\nu$ is compactly contained in $X_\ll^2$. 
\end{enumerate}
\end{defn}

\begin{prop}[Stability of optimal couplings for totally timelike pairs]\label{Stableopt}
Let $(X,\ell,\tau, \m)$ be a measured topological causal space with $\ell_+$ continuous. Let $\{(\mu^n_0, \mu^n_1)\}$ be a sequence of totally timelike pairs of measures which are concentrated on the same compact set $K\subset X_\ll^2$.  Then $(\mu^n_0, \mu^n_1)$ has a narrow subsequential limit $(\mu_0, \mu_1)$, and any sequence $\pi^n\in \Opt_q(\mu_0^n, \mu_1^n)$ has a further narrow subsequential limit $\pi \in \TOpt_q(\mu_0, \mu_1)$. 
\end{prop}
\begin{proof}
Compactness of $K$ implies tightness of $\{(\mu_0^n, \mu_1^n)\}_{n\in \N}$,  and so the set admits some subsequential limit $(\mu_0, \mu_1)$ which is also concentrated on $K$. Moreover tightness implies that any sequence $\{\pi^n\}_{n\in\N}$ such that $\pi^n\in \Opt_q(\mu_0^n, \mu_1^n)$ has a further subsequential limit $\pi\in \Pi(\mu_0,\mu_1)$. As each pair $(\mu_0^n, \mu_1^n)$ is totally timelike, each $\pi^n$ is concentrated on an $\ell_+^q$-cyclically monotone set contained in $X_\ll^2$, and since $\ell_+$ is continuous, this property is inherited in the (subsequential) limit. Therefore $\pi$ is concentrated on a $\ell_+^q$-cyclically monotone set and concentrated on $K\subset X_\ll^2$, and therefore by \Cref{cyclicallymonotoneoptimal} is $q$-optimal. 
\end{proof}

\begin{prop}[Restriction preserves optimality]
Suppose $\pi\in\Pcal(X^2)$ is timelike $q$-optimal between its marginals. For each $f\in L^1(\pi)$ with $ \|f\|_{L^1(\pi)}=1$, the measure $\pi_f$ defined by $d\pi_f:=fd\pi$ is also timelike $q$-optimal between its marginals. 
\end{prop}
\begin{proof}
Evidently the measure $\pi_f$ is concentrated on the same set as $\pi$. As $\pi$ is timelike and $q$-optimal, by \Cref{cyclicallymonotoneoptimal} is concentrated on an $\ell_q$-cyclically monotone set $\Gamma\subset X_\ll^2$, and therefore so is $\pi_f$. By \Cref{cyclicallymonotoneoptimal}, $\pi_f$ is $q$-optimal between its marginals. 
\end{proof}

\subsubsection{Optimal Plans and Geodesics}

Given $t\in [0,1]$, we define the {\it $t$-evaluation map} $e_t: C([0,1],X)\to X$ given by  $e_t(\gamma):=\gamma_t$. We also define the space $\Geo(X)\subset C([0,1],X)$ to be the set of geodesics parametrized in $[0,1]$. 

\begin{defn}[Optimal plans and displacement geodesics]
A measure $\eta\in \Pcal(\Geo(X))$ is a {\it $q$-optimal plan} between $(\mu,\nu)\in\Pcal(X)^2$ if $(e_0\times e_1)_\# \eta\in \Opt_q(\mu,\nu)$. The set of $q$-optimal plans from $\mu$ to $\nu$ is denoted $\OptGeo_q(\mu,\nu)$.  A plan $\eta$ is {\it timelike} if it is concentrated on $X_\ll^2$, and the set of timelike $q$-optimal plans is denoted $\TOptGeo_q(\mu,\nu)$. Note that if $\eta$ is $q$-optimal plan then $\mu_t:= (e_t)_\#\eta$ is a geodesic. Such a geodesic is called a {\it displacement geodesic}. 
\end{defn}
{\bf Remark:} (On definitions) What we call an ``optimal coupling'', other authors call an ``optimal plan'', and what we call an 		``optimal plan'', they call an {\it ``optimal dynamical plan}''. 

\begin{prop}[Existence of optimal plans and displacement geodesics, \text{\cite[Thm.\ 2.43]{Octet}}]\label{optimalplanprop}
Let $(X,\ell, \tau)$ be a compact synthetic spacetime. Then every $q$-optimal coupling is induced by a $q$-optimal plan, and every geodesic on $\Pcal(X)$ is a displacement geodesic. 
\end{prop}

\begin{defn}[Timelike ($q$-essential) non-branching]\label{nonbranchingdef}
A set $A\subset X$ is {\it timelike non-branching} if for each $0\leq s\leq t\leq 1$, the restriction map $r_s^t: C([0,1],A)\to C([0,1],A)$ given by $(r_s^t(\gamma))_u=\gamma_{(1-u)s+ut}$ is injective on the set of timelike curves in $A$. We say that $X$ is {\it timelike $q$-essentially non-branching} if every timelike optimal plan $\eta\in \TOptGeo_q(\mu_0,\mu_1)$ is concentrated on a timelike non-branching set, for all $\mu_0,\mu_1\in\Pcal(X)$.
\end{defn}

The non-branching assumption prevents optimal plans from ``meeting'' in a sense outlined by the following lemma. This results mirrors now classic results in the metric measure setting; \cite[Lem.\ 2.6]{BacherSturm}, and \cite[Lem.\ 2.11]{SturmII}. In the Lorentzian setting we point out \cite[Lem.\ 2.23]{Braun}.
\begin{defn}[Midpoint]
Fix $q\in (0,1)$. Let $\mu_0,\mu_1\in (\Pcal(X), \ell_q)$, and let $t\in [0,1]$. A measure $\nu\in(\Pcal(X), \ell_q)$ is called a $t$-midpoint between $\mu_0$ and $\mu_1$ if 
\[\ell_q(\mu_0, \nu)=t\ell_q(\mu_0, \mu_1); \quad \ell_q(\nu,\mu_1)=(1-t)\ell_q(\mu_0, \mu_1).\]
\end{defn}

\begin{lemma}[Midpoints between mutually singular measures remain mutually singular]\label{singularmidpoint}
Suppose $(X,\ell, \tau, \m)$ is a timelike $q$-essentially non-branching globally hyperbolic synthetic spacetime with $\ell_+$ continuous. Fix $t\in [0,1)$. Let $\{(\mu_{0k}, \mu_{1k})\}_{k=1}^n$ be a finite collection of $q$-timelike pairs of measures in $\Pcal_c(X)^2$, and let $\mu_{tk}$ be a $t$-midpoint joining $\mu_{0k}$ to $\mu_{1k}$. For $i=0,t,1$, let $\mu_{i}:=\sum_{k=1}^n \lambda_k\mu_{ik}\in \Pcal(X)$. Suppose $(\mu_0, \mu_1)$ is a $q$-timelike pair, and $\mu_t$ is a $t$-midpoint between $\mu_0$ and $\mu_1$. If  $\mu_{0k}\perp \mu_{0k'}$ for all $k\neq k'$, then $\mu_{t k}\perp \mu_{t k'}$ for all $k\neq k'$. 
\end{lemma}
\begin{proof}

\newcommand{\glu}{\operatorname{glu}}
Fix $t\in (0,1)$. 
Let $Y:=J(\spt\mu_0,\spt\mu_1)$. From \Cref{spacetimestructure}, $\Pcal(Y)$ is a compact synthetic spacetime, and WLOG we may restrict ourselves to $\Pcal(Y)$ as every relevant measure lives in this space. We have that $\mu_t:= \sum_{k=1}^n \lambda_k \mu_{t,k}$ is a $t$-midpoint between $\mu_0$ and $\mu_1$. As $\Pcal(Y)$ is causally geodesic, for each $k=1,\dots, n$, there are geodesics $s\mapsto \mu^-_{s,k}$ and $s\mapsto \mu^+_{s,k}$ which join $\mu_{0,k}$ to $\mu_{t,k}$, and $\mu_{t,k}$ to $\mu_{1,k}$, respectively. Define $s\mapsto \mu_{s,k}$ via 
\[\mu_{s,k}:= \begin{cases}
\mu_{s/t,k}^- & 0\leq s\leq t;\\
\mu_{(s-t)/(1-t),k}^+ &t\leq s\leq 1.
\end{cases}\]

It's a straightforward computation to show that $\mu_{t,k}$ is a geodesic path joining $\mu_{0,k}$ to $\mu_{1,k}$, which from \Cref{geodesicpathtocurve} and \Cref{spacetimestructure} implies that $\mu_{t,k}$ is a geodesic. Then \Cref{optimalplanprop} implies that $\mu_{t,k}$ is induced by a timelike $q$-optimal plan $\eta_{k}$. 

Since $\mu_t:= \sum_{k=1}^n \lambda_k\mu_{t,k}$ is also a $t$-midpoint, by the same proof we have a geodesic $s\mapsto \mu_s$ joining $\mu_0$ to $\mu_1$ which meets $\mu_t$ at $s=t$, such that $\mu_s$ is induced by the timelike $q$-optimal plan $\eta:=\sum_{k=1}^n \lambda_k \eta_k$. 

Now suppose there is $k_1\neq k_2$ such that $\mu_{t,k_1}\not\perp \mu_{t,k_2}$. For $i=0,t,1$, we define the measures
\[\nu_i:= \frac{\mu_{i,k_1}+\mu_{i,k_2}}{2},\]
with associated timelike $q$-optimal plan 
\[\zeta:= \frac{\eta_{k_1}+\eta_{k_2}}{2}.\] 
Recall the restriction map $r_s^t:\Geo(Y)\to \Geo(Y)$ (see \Cref{nonbranchingdef}), and define 
\begin{align*}
\zeta^-&:= (r_0^t)_\# \zeta,\\
\zeta^+&:= (r_t^1)_\#\zeta.
\end{align*}

By disintegration with respect to the evaluation maps $e_0,e_1: \Geo(Y)\to Y$, we obtain the families $\{\zeta^{\pm}_x\}_{x\in Y}$ of measures defined via 
\[d\zeta^{\pm}(\gamma) = d\zeta^\pm_x(\gamma)d\nu_t(x).\]

We define the gluing map $\glu:Glu\to C([0,1],Y)$ on the set $Glu:=\{(\gamma,\sigma)\in \Geo(Y)^2: \gamma_1=\sigma_0\}$. The gluing map is continuous, so for $\nu_t$-a.e $x$, the measure
\[\zeta^M_x := \glu_\# (\zeta^-_x\otimes \zeta^+_x),\]
is well defined. Therefore we may define the mixed plan $\zeta^M\in \Pcal(C([0,1],X))$ via 
\[d\zeta^M(\gamma)= \zeta^M_x(\gamma)d\nu_t(x).\]

We claim that $\zeta^M$ is a $q$-optimal plan between $\nu_0$ and $\nu_1$, and yet cannot be concentrated on a timelike non-branching set. The argument is based on Steps 2 and 3 of the proof of \cite[Lem.\ 2.23; pp 68-70]{Braun}, which itself is more or less identical to \cite[\S 4]{RajalaSturm}. Within those proofs is the assumption that the t-midpoints $\nu_t$ are $\m$-absolutely continuous, which we remove. Apart from this, we also struggle to understand the finer details of the proofs in the above references, and so we attempt to provide our own below. 
\paragraph{Optimality of $\zeta^M$:}

We must show that $\zeta^M\in \Pcal(\Geo(X))$, and 
\[\ell_q^q(\nu_0,\nu_1)=\int \ell^q(\gamma_0, \gamma_1)d\zeta^M(\gamma). \]

As $\zeta$ is a timelike $q$-optimal plan, there is an $\ell^q$-cyclically monotone set $\Gamma\subset Y_\ll^2$ such that $\zeta$ is concentrated on $\hat\Gamma:= (e_0\times e_1)^{-1}(\Gamma)\subset \Geo(Y)$. Let $\gamma,\sigma\in \hat\Gamma$ be such that $\gamma_\lambda=\sigma_\lambda$. Then
\begin{align}\label{constantage}
\begin{split}
\ell^q(\gamma_0, \gamma_1)+\ell^q(\sigma_0, \sigma_1) &\geq \ell^q(\gamma_0, \sigma_1)+\ell^q(\sigma_0, \gamma_1) \geq \left(\ell(\gamma_0, \gamma_\lambda) +\ell(\sigma_\lambda, \sigma_1)\right)^q+ \left(\ell(\sigma_0, \sigma_\lambda) +\ell(\gamma_\lambda, \gamma_1)\right)^q \\
&= \left(\lambda\ell(\gamma_0, \gamma_1) +(1-\lambda)\ell(\sigma_0, \sigma_1)\right)^q+ \left(\lambda\ell(\sigma_0, \sigma_1) +(1-\lambda)\ell(\gamma_0, \gamma_1)\right)^q\\
&\geq \lambda\ell^q(\gamma_0, \gamma_1)+(1-\lambda)\ell^q(\sigma_0, \sigma_1)+ \lambda\ell^q(\sigma_0, \sigma_1)+(1-\lambda)\ell^q(\gamma_0, \gamma_1)\\
&= \ell^q(\gamma_0, \gamma_1)+\ell^q(\sigma_0, \sigma_1),
\end{split}
\end{align}
which gives equality in all above inequalities. This implies that $\ell(\gamma_0, \gamma_1)=\ell(\sigma_0,\sigma_1)=\ell(\gamma_0,\sigma_1)=\ell(\sigma_0,\gamma_1)$.

Observe that for $\nu_t$-a.e.\ $x\in Y$, $\zeta^-_x$  (resp. $\zeta^+_x$) is concentrated on $r_0^t(\hat\Gamma)\cap e_1^{-1}(x)$ (resp. $r_t^1(\hat\Gamma)\cap e_0^{-1}(x)$). Therefore for $\nu_t$-a.e.\ $x$, $\zeta^M_x$ is concentrated on the set $S(x):=\{\gamma\in C([0,1], Y): r_0^t\gamma \in r_0^t (\hat \Gamma), r_1^t\gamma \in r_1^t(\hat\Gamma), \gamma_t=x\}$. The claim is that each $\gamma\in S(x)$ is a geodesic whose age $\age(\gamma)= \ell(\gamma_0,\gamma_1)$ is a constant $c(x)>0$ depending, a priori, on $x$ only.  

Indeed, let $\gamma\in S(x)$. Then there is $\sigma, \sigma'\in \hat\Gamma$ such that $r_0^t\gamma=r_0^t\sigma$ and $r_1^t\gamma= r_1^t\sigma'$. Since $\gamma_t=x$, we have that $\sigma_t=\sigma_t'=x$. Then from \eqref{constantage}, we see that $\ell(\sigma_0,\sigma_1)=\ell(\sigma'_0, \sigma'_1)= \ell(\sigma_0, \sigma'_1)=\ell(\sigma'_0,\sigma_1)$. \\

Now let $s,s'\in [0,1]$, $s<s'$. If $s'\leq t$, then 

\[\ell(\gamma_s,\gamma_{s'})= \ell(\sigma_{s}, \sigma_{s'})= (s'-s)\ell(\sigma_0, \sigma_1)=(s'-s)\ell(\sigma_0,\sigma_1')=(s'-s)\ell(\gamma_0,\gamma_1).\]
Similarly if $t\leq s$, we obtain $\ell(\gamma_s,\gamma_{s'})=(s'-s)\ell(\gamma_0,\gamma_1)$. Finally in the case $s\leq t\leq s'$, we obtain
\[\ell(\gamma_s,\gamma_{s'})= \ell(\sigma_s, \sigma'_{s'})\geq \ell(\sigma_s, \sigma_t)+\ell(\sigma'_t,\sigma'_{s'}) = (t-s)\ell(\sigma_0, \sigma_1)+(s'-t)\ell(\sigma'_{0},\sigma'_{1})= (s'-s)\ell(\gamma_0,\gamma_1).\]
This proves that $\gamma$ is a geodesic. Now let $\bar\gamma\in S$. Then there is $\bar\sigma,\bar \sigma'\in \hat\Gamma$ such that $r_0^t\bar\gamma= \bar \sigma$ and $r_t^1\bar\gamma=\bar\sigma'$, and moreover comparing the ages of the four geodesics $\sigma,\sigma',\bar\sigma, \bar\sigma'$ using \eqref{constantage}, we see that 
\[\ell(\gamma_0,\gamma_1)=\ell(\sigma_0,\sigma'_1)= \ell(\sigma_0, \bar\sigma'_1)= \ell(\bar\sigma_0, \bar\sigma'_1)=\ell(\bar\gamma_0,\bar\gamma_1), \]
hence there is some constant $c(x)>0$ depending only on $x$ such that $\ell(\gamma_0,\gamma_1)=c(x)$ for all $\gamma\in S$. In particular we have that $\zeta^M_x\in \Pcal(\Geo(Y))$ for all $\nu_t$-a.e\ $x\in Y$, and so $\zeta^M\in \Pcal(\Geo(X))$. 

Disintegrating $d\zeta (\gamma)= \zeta_x(\gamma)d\nu_t$, one can see, by the same proof as above, that $\zeta_x$ is also concentrated on the set of geodesic of constant age $c(x)$, for $\nu_t$-a.e\ $x\in Y$. 

Now observe 
\begin{align*}
 \int \ell^q(\gamma_0,\gamma_1) d\zeta^M(\gamma) &= \int \int \ell^q(\gamma_0,\gamma_1)d \zeta^M_x(\gamma)d\nu_t(x) = \int c^q(x)d\nu_t(x)=\int c^q(x) d\zeta_x(\gamma) d\nu_t(x)\\
 &=\int\int \ell^q(\gamma_0,\gamma_1)d \zeta_x(\gamma)d\nu_t(x)=\int \ell^q(\gamma_0,\gamma_1) d\zeta(\gamma) = \ell_q^q(\nu_0,\nu_1).\end{align*}

Thus $\zeta^M$ is $q$-optimal. 

\paragraph{Branchingness of $\zeta^M$:}
Let $C\subset \Geo(Y)$ be a set on which $\zeta^M$ is concentrated. We will show that $C$ is timelike branching. 

For each $i=1,2$, we define $\eta^-_{k_1}= (r_0^t)_\# \eta_{k_i}$. We then disintegrate $d\eta^-_{k_i} = (\eta^-_{k_i})_x d\mu_{t,k_i}$. Since $\mu_{t,k_1}\not\perp \mu_{t,k_2}$, we can find a Borel set $E\subset Y$ of positive $\nu_t$-measure such that $(\eta^-_i)_x$ is concentrated on $e_1^{-1}(x)$ for all $x\in E$. We may write 
\[\zeta^-_x= \frac 12(\eta^-_{k_1})_x+\frac 12 (\eta^-_{k_2})_x.\]
 As $\mu_{0,k_1}\perp \mu_{0,k_2}$, we have that $\eta^-_{k_1}\perp \eta^-_{k_2}$, and therefore $(\eta^-_{k_1})_x\perp (\eta^-_{k_2})_x$.  We then have 
 \[\zeta^-_x\otimes \zeta^+_x = \frac 12(\eta^-_{k_1})_x\otimes \zeta^+_x+\frac12(\eta^-_{k_2})_x\otimes \zeta^+_x \]
 
 Since $\zeta^M$ is concentrated on $C$, $\zeta^M_x$ is concentrated on $C$, and $\zeta^-_x\otimes \zeta^+_x $ is concentrated on $D:= \glu^{-1}(C)\subset C([0,1],Y)^2$. Moreover $(\eta^-_{k_1})_x\perp (\eta^-_{k_2})_x$ implies that $(\eta^-_{k_1})_x\otimes \zeta^+_x\perp (\eta^-_{k_2})_x\otimes \zeta^+_x$. Therefore there is a disjoint partition $A\sqcup B = D$ such that $(\eta^-_{k_1})_x\otimes \zeta^+_x$ concentrates on $A$, and $(\eta^-_{k_2})_x\otimes \zeta^+_x$ concentrates on $B$. Let $p_2:C([0,1],Y)^2\to C([0,1], Y)$ denote the projection map onto the second coordinate, $p_2(\gamma, \sigma)=\sigma$.  Note that for both $i=1,2$, 
 \[\zeta_x^+=(p_2)_\# ((\eta^-_{k_i})_x\otimes \zeta^+_x).\]
It follows that $\zeta_x^+$ is concentrated on both $p_2(A)$ and $p_2(B)$, so it cannot be that these two sets are disjoint. But $A$ and $B$ are disjoint, so therefore there are three geodesics $\gamma,\gamma', \sigma$ such that $(\gamma,\sigma)\in A$, $(\gamma',\sigma)\in B$. Then $\glu(\gamma,\sigma_1), \glu(\gamma',\sigma)\in C$, but the curves $\glu(\gamma, \sigma)$, $\glu(\gamma', \sigma)$ are two distinct geodesics which restrict to the same geodesic $\sigma$. Thus $C$ is timelike branching. 
\end{proof}

\subsubsection{The Spacetime Structure of $\Pcal(X)$}
It is useful to know when the spacetime structure of $X$ is inherited by $\Pcal(X)$. The following lemma shows that compactness is enough to guarantee this. 

\begin{lemma}\label{spacetimestructure}
Suppose $(X, \ell, \tau) $ is a compact synthetic spacetime. Then $(\Pcal(X), \ell_q)$ is a compact synthetic spacetime for any $q\in(0,1)$.
\end{lemma}
\begin{proof}
As $X$ is compact, it is well-known that $\Pcal(X)$ is compact in the narrow topology. Therefore we need only prove the following properties: 
\begin{enumerate}
\item The narrow topology is Polish;
\item The time separation $\ell_q$ is upper semicontinuous and does not attain $+\infty$; 
\item The space $\Pcal(X)$ is causally geodesic.
\end{enumerate}

{\it  Proof of 1).} Since $X$ is compact, the narrow topology is equivalent to the topology generated by the Polish metric $W_p$ for any $p\geq 1$. \\

{\it Proof of 2).} Since $X$ is compact, and $\ell$ is upper semicontinuous and does not attain infinity, $\ell$ attains an upper bound $M$. But then $\ell_q$ attains the same upper bound. 

Let $\{(\mu_n, \nu_n)\}_{n\in\N}\subset \Pcal(X)^2$ be a sequence of pairs of measures which narrowly converge to $(\mu,\nu)$.  Suppose by way of contradiction that 
\[\limsup_{n\to\infty}\ell_q^q(\mu_n,\nu_n) >\ell_q^q(\mu,\nu).\]
Then there is a $\{(\mu_{n_k}, \nu_{n_k})\}_{k\in\N}$ such that for all $k\in\N$, 
\begin{align*}
\ell_q^q(\mu_{n_k},\nu_{n_k})> \ell_q^q(\mu,\nu).
\end{align*}

Since $X$ is compact, each $(\mu_{n_k},\nu_{n_k})$ is dualizable. Hence there is $\pi_{n_k}\in \Opt_q(\mu_{n_k},\nu_{n_k})$. Tightness of $\{(\mu_{n_k}, \nu_{n_k})\}_{k\in\N}$ implies tightness of $\{\pi_{n_k}\}_{k\in\N}$.
Then, after passing to a subsequence and relabeling, there exists a limit coupling $\pi\in\Pcal(X\times X)$. Narrow continuity of the pushforwards of the projection maps, $(p_i)_\#$, for $i=0,1$ show that $\pi\in \Pi(\mu,\nu)$. Using that $\ell$ is upper semicontinuous, and thus also bounded above on the compact set $X^2$, we see from Portmanteau's theorem that 
\[\ell_q^q(\mu,\nu)< \limsup_{k\to\infty} \ell_q^q(\mu_{n_k}, \nu_{n_k})=\limsup_{k\to\infty}\int \ell^q d\pi_{n_k}\leq \int \ell^q d\pi \leq \ell_q^q(\mu,\nu), \]
a contradiction. Therefore $\ell_q$ is upper semicontinuous.\\

{\it Proof of 3).} Let $\mu,\nu\in \Pcal(X)$ be such that $\ell_q(\mu,\nu)\geq0$. Then there exists a $q$-optimal coupling $\pi\in \Opt_q(\mu,\nu)$ which is concentrated on $X_\leq^2$. From \Cref{optimalplanprop}, $\pi$ induces a $q$-optimal plan $\eta\in \OptGeo_q(\mu,\nu)$, and thus a geodesic $\mu_t:= (e_t)_\#\eta$ joining $\mu$ to $\nu$.  

\end{proof}
\subsection{Timelike Curvature-Dimension Conditions} 

We now present the timelike curvature conditions, introduced by McCann \cite{McCann}, and independently, Mondino--Suhr \cite{MS} in the smooth setting, and Cavalletti--Mondino \cite{CM} and Braun \cite{Braun}, in the non-smooth setting; they are Lorentzian analogues of the metric notions introduced in \cite{LV, SturmI, SturmII, EKS}. 

\subsubsection{Volume Distortion Coefficients}
\begin{defn}[Distortion coefficients] Let $k\in\R$. The {\it $k$-generalized sine function} $\sin_k(t)$ is the solution to the initial value problem 
\begin{align}
\begin{cases}
f''(t)+kf(t)=0;\\
f(0)=0;\\
f'(0)=1.
\end{cases}
\end{align}
For $t\in [0,1]$ and $\theta\geq 0$, the {\it reduced distortion coefficient} $\sigma_{k}^{(t)}(\theta)$ is defined as 
\begin{align}
\sigma_k^{(t)}(\theta):=\begin{cases}
\displaystyle \frac{\sin_k(t\theta)}{\sin_k(\theta)} & k\theta^2<\pi^2;\\
+\infty &\text{otherwise.}
\end{cases}
\end{align}

For $N>1$ and $K\in \R$ the {\it distortion coefficient} $\tau_{K,N}^{(t)}(\theta)$ is defined as 
\begin{align}
\tau_{K,N}^{(t)}(\theta):= t^{1/N}\left(\sigma_{K/(N-1)}^{(t)}(\theta)\right)^{1-1/N},
\end{align}
with the convention that $(\infty)^{1-1/N}=\infty$. 
\end{defn}

The reduced distortion coefficients are solutions to the differential equation $f''+k\theta^2f=0$ with boundary conditions $f(0)=0$ and $f(1)=1$. In fact, any $C^2$ function $f:I\to [0,\infty)$ solves $f''+kf\leq 0$ over an interval $I\subset \R$ iff 
\begin{align}\label{kconvexity1}
f((1-t)a+tb)\geq \sigma^{(1-t)}_k(|b-a|)f(a)+\sigma_k^{(t)}(|b-a|)f(b)\quad \forall a,b\in I.
\end{align}

If, moreover $f$ can be decomposed as $f^N=f_\|f_\perp^{N-1}$, with $f_\|''\leq 0$ and $f_\perp''+\frac{K}{N-1}f_\perp\leq0$, then $f$ satisfies 
\begin{align}\label{kconvexity2}
f((1-t)a+tb)\geq \tau^{(1-t)}_{K,N}(|b-a|)f(a)+\tau_{K,N}^{(t)}(|b-a|)f(b)\quad \forall a,b\in I.
\end{align}

Inequality \eqref{kconvexity2} is stronger than \eqref{kconvexity1} (with $k:=K/N$), in light of property 2 of the following proposition.
\begin{prop}[Properties of distortion coefficients, \text{\cite[Lem.\ 1.2]{SturmII}}]
Let $N> 1$ and $K\in \R$. 
\begin{enumerate}
\item For each $a,\theta\geq 0$ and each $t\in [0,1]$, $\sigma_{K}^{(t)}(a\theta)= \sigma_{a^2K}^{(t)}(\theta)$; 
\item For each $\theta\geq 0$ and each $t\in[0,1]$, $\sigma_{K/N}^{(t)}(\theta)\leq \tau_{K,N}^{(t)}(\theta)$;
\item For each $\theta\geq 0$ and $t\in[0,1]$, the map $K\mapsto \tau_{K,N}^{(t)}(\theta)$ is continuous, log-convex, and non-decreasing on the domain $\{K\in \R: K\theta^2<\pi^2\}$.
\end{enumerate}
\end{prop}

Properties 1 and 3 above give the following corollary, which will be useful for us later on. 
\begin{corollary}\label{logconvexityoftau}
For each $N>1$ and $K\in \R$, the map $\theta\mapsto \tau_{K,N}^{(t)}\big(\sqrt\theta\big)$ is continuous and log-convex on the domain $\{\theta\geq 0: K\theta^2<\pi^2\}$.
\end{corollary}
\subsubsection{Entropies and Curvature Conditions}
\begin{defn}[R\'enyi and Boltzmann entropy] Let $N\in [0,\infty)$. The {\it $N$-R\'enyi entropy} $S_N:\Pcal(X)\to [-\infty,0]$ is defined as 
\begin{align}
S_N(\mu):= -\int \left(\diff{\mu}{\m}\right)^{1-1/N}d\m,
\end{align}
where $\diff{\mu}{\m}$ is the Lebesgue-Radon-Nikodym derivative of $\mu$ with respect to $\m$. The {\it Boltzmann-Shannon entropy} $\Ent: \Pcal(X)\to [0,\infty]$ is defined as 
\begin{align}
\Ent(\mu):=\begin{cases}
\displaystyle \int \rho \log\rho d\m & \text{if $d\mu = \rho d\m$;}\\
+\infty &\text{otherwise.}
\end{cases}
\end{align}
\end{defn}
Using the entropies, we define the synthetic curvature conditions $\TMCP(K,N)$ and $\TCD_q(K,N)$, as well as their entropic counterparts. For the entropic versions, we consider the exponential of the entropy
\[U_N(\mu):=\exp\left(-\frac{1}{N}\Ent(\mu)\right).\]
\begin{defn}[Curvature-dimension conditions] Let $(X,\ell, \tau, \m)$ be a measured synthetic spacetime, and let $N>1$, $K\in\R$. We say $X$ satisfies... 
\begin{enumerate}
\item ... the {\it $q$-timelike $(K,N)$-curvature-dimension condition} $\TCD_q(K,N)$ if every $q$-timelike pair $(\mu_0,\mu_1)\in \Pcal_c^{ac}(X)^2$ admits a timelike $q$-optimal plan $\eta$ such that for all $N'>N$ and all $0\leq t\leq 1$,
\begin{align}
S_{N'}((e_t))_\# \eta)\leq -\int_{X\times X} \Big(\tau_{K,N'}^{(1-t)}(\ell(\gamma_0,\gamma_1)) \rho_0(\gamma_0)^{-1/N'}+\tau_{K,N'}^{(t)}(\ell(\gamma_0,\gamma_1)) \rho_1(\gamma_1)^{-1/N'}\Big)d\eta(\gamma),
\end{align}
where $d\mu_i=\rho_id\m$ for $i=0,1$. 
\item ... the  {\it timelike $(K,N)$-measure contraction property} $\TMCP(K,N)$, if every totally timelike pair $(\mu,\delta_{x_0})\in\Pcal_c^{ac}(X)\times \Pcal_c(X)$ induces a timelike $q$-optimal plan $\eta$ such that for all $N'>N$, and all $0\leq t\leq 1$,  
\begin{align}
S_{N'}((e_t)_\#\eta)\leq -\int_{X\times X}\tau_{K,N'}^{(1-t)}(\ell(x,x_0)) \rho(x)^{-1/N'}d\mu(x),
\end{align}
where $d\mu=\rho d\m$. 
\item ... the {\it entropic $q$-timelike $(K,N)$-curvature-dimension condition} $\TCD^e_q(K,N)$ if every $q$-timelike pair $(\mu_0,\mu_1)\in \Pcal_c^{ac}(X)^2$ admits a timelike $q$-optimal plan $\eta$ such that for all $0\leq t\leq 1$,
\begin{align}
U_N((e_t)_\#\eta)\geq \sigma_{K/N}^{(1-t)}(\Lambda)U_N(\mu_0)+\sigma_{K/N}^{(t)}(\Lambda)U_N(\mu_1),
\end{align}
where 
\[\Lambda:= \left(\int_{X\times X} \ell^2(\gamma_0,\gamma_1)d\eta(\gamma)\right)^{1/2},\]

\item ...  the {\it entropic timelike (K,N)-measure contraction property} $\TMCP^e(K,N)$, if every totally timelike pair $(\mu,\delta_{x_0})\in \Pcal_c^{ac}(X)\times \Pcal_c(X)$ induces a timelike $q$-optimal plan $\eta$ such that for all $0\leq t\leq 1$, 
\begin{align}
U_N((e_t)_\#\eta)\geq\sigma_{K/N}^{(1-t)}(\Lambda) U_N(\mu).
\end{align}
\end{enumerate}
\end{defn}

We now introduce the timelike strong Brunn Minkowski inequality. Given a Borel set $A\subset X$ of positive finite $\m$-measure, we write $\m_A$ for the uniform measure on $A$, i.e $\m_A:= \m(A)^{-1}\m\vert_A$. We say a pair of Borel sets $(A,B)$ is {\it $q$-timelike} if both sets are of positive finite $\m$-measure, and the pair of uniform measures $(\m_A,\m_B)$ is $q$-timelike.

\begin{defn}[Timelike strong Brunn-Minkowski inequality]
Let $X=(X,d,\ell, \m)$ be a measured $\mathcal{K}$-globally hyperbolic spacetime. Let $q\in(0,1)$, $K\in \R$, and $N>1$. 
We say $X$ satisfies... 

\begin{enumerate}
\item ... the {\it timelike $(K,N)$-Brunn-Minkowski inequality}, $\TBM(K,N)$, if for any pair $(A,B)$ of compact totally timelike sets,  
\begin{align}
\m^{1/N}(G_t(A,B))\geq \tau^{(1-t)}_{K,N}(\Theta)\m^{1/N}(A)+\tau^{(t)}_{K,N}(\Theta)\m^{1/N}(B),
\end{align}
where 
\[G_t(A,B):= \{\gamma_t: \gamma\in \Geo(X), (\gamma_0,\gamma_1)\in (A\times B)\cap X_\ll^2\},\]
and 
\[\Theta:= \begin{cases}
\inf_{A\times B} \ell(x,y) & K\geq 0; \\
\sup_{A\times B} \ell(x,y) & K<0.
\end{cases}\]

\item ... the {\it $q$-timelike strong $(K,N)$-Brunn-Minkowski inequality}, $\sTBM_q(K,N)$, if for any pair $(A,B)$ of compact $q$-timelike sets, there exists a timelike $q$-optimal plan $\eta$ such that 
\begin{align}
\m^{1/N}(\spt(e_t)_\#\eta)\geq \tau^{(1-t)}_{K,N}(\Theta)\m^{1/N}(A)+\tau^{(t)}_{K,N}(\Theta)\m^{1/N}(B).
\end{align}

\item ... the {\it reduced $\sTBM_q(K,N)$ inequality}, $\sTBM_q^*(K,N)$, if for any pair $(A,B)$ of compact $q$-timelike sets, there exists a timelike $q$-optimal plan $\eta$ such that 
\begin{align}
\m^{1/N}(\spt(e_t)_\#\eta)\geq \sigma^{(1-t)}_{K/N}(\Theta)\m^{1/N}(A)+\sigma^{(t)}_{K/N}(\Theta)\m^{1/N}(B).
\end{align}
\end{enumerate}
\end{defn}

\subsection{Simple Measures and Simple Sequences}
 Observe that for any uniform measure $\m_A$, $-S_N(\m_A)=U_N(\m_A)= \m^{1/N}(A)$. One can therefore think of the $\sTBM_q(K,N)$ condition as a curvature condition at the level of uniform measures. It's perhaps straightforward to see then that $\TCD_q(K,N)$ implies $\sTBM_q(K,N)$, since uniform measures are in $\Pcal_c^{ac}(X)$. To prove the converse amounts to showing that {\it any} probability measure in $\Pcal_c^{ac}(X)$ can be approximated in some way by uniform measures. At the level of the densities of the measures in $\Pcal_c^{ac}(X)$, this amounts to approximating $L^1(\m)$ functions by simple, that is, piecewise constant, functions.  

\begin{defn}[Simple measures]
We call a measure $\mu\in \Pcal(X)$ {\it simple} if it is a finite convex combination of mutually singular uniform measures. That is, there exists $M\in \N$, mutually disjoint Borel sets $\{A_j\}_{j=1}^n$ of positive finite $\m$-measure, and positive numbers $\{\lambda_j\}_{j=1}^n$ such that $\sum_{j=1}^n \lambda_j=1$, and 
\begin{align*}
\mu= \sum_{j=1}^n \lambda_j \m_{A_j}.
\end{align*}
\end{defn}

We want to be able to approximate measures $\mu\in \Pcal_c^{ac}(X)$ by a sequence of simple measures $\{\mu^n\}_{n\in\N}$. Later on we will want to choose sequences such that a $q$-timelike pair $(\mu_0,\mu_1)$ can be approximated by a sequence of pairs of simple measures $(\mu_0^n, \mu_1^n)$ which are ``correlated'' in some precise way. To do this we require existence and uniqueness of optimal maps, so we delay this lemma for later. For now we prove that any $\mu\in \Pcal_c^{ac}(X)$ can be approximated by a sequence of simple measures with some additional desirable properties. Such a sequence we call a {\it simple sequence}. 
 
\begin{prop}[Existence of simple sequences]\label{simplesequence} Let $\mu\in \Pcal_c^{ac}(X)$. Then there exists a sequence $\{\mu^n\}_{n\in\N}$ of simple measures, which we call a ``simple sequence", which satisfies the following properties: 
\begin{enumerate}
\item (Narrow Convergence). As $n\to\infty$, $\mu^n\narrowto\mu$. 
\item (Controlled Decomposition). For each metrizer $d$ of $X$, and each positive sequence $\delta_n\to 0$, we may write $\mu^n:= \sum_{j=1}^{M_n}\lambda_j^n \m_{A_j^n}$, where for each $n\in\N$ and each $j=1,\dots, M_n$, $A_j^n$ is compact and contained in $\spt\mu_0$, and $\diam_d(A_j^n)<\delta_n$. 

\item (Continuity along Entropy) For each $N\in\N$, $U_N(\mu^n)\to U_N(\mu)$, and $S_N(\mu^n)\to S_N(\mu)$. 

\end{enumerate}
\end{prop}
\begin{proof}
Let $\mu\in \Pcal_c^{ac}(X)$, and let $\rho\in L^1(\m)$ be such that $d\mu=\rho d\m$. There exists an increasing sequence $\bar\rho^n$ of simple functions which converge pointwise to $\rho$. Observe then that $\|\bar\rho^n\|_{L^(\m)}\to \|\rho\|_{L^1(\m)}=1$. We define $\tilde\mu^n$ via
\[d\tilde \mu^n:=\frac{\bar\rho^n}{\|\rho^n\|_{L^1(\m)}}d\m.\]
It's clear that $\tilde\mu^n$ is a simple measure for each $n$, and so we may write
\[\tilde\mu^n= \sum_{j=1}^{M_n} \lambda_j^n \m_{\tilde A_j^n}. \]
where $\{\tilde A_j^n\}_{j=1}^{M_n}$ is a collection of disjoint Borel sets of positive and finite $\m$-measure. Since $\bar \rho^n\leq \rho$, we have that $\tilde A_j^n\subset \spt\mu_0$ for each $n\in\N$ and each $j=1,\dots, M_n$. Since $\spt\mu_0$ is compact, for any metrizer $d$ and any positive sequence $\delta^n$, we may intersect each element of $\{\tilde A_j\}_{j=1}^{M_n}$ by finitely many open $d$-balls of diameter $\delta^n$, and so WLOG $\diam_d(\tilde A_j^n)<\delta^n$.

We now prove convergence in entropy. By construction for each $\epsilon>0$, there is $n_0\in\N$ such that for all $n>n_0$,
\[\rho-\epsilon\leq \tilde\rho^n\leq \rho.\]
Then observe that 
\begin{align*}
\rho^n\log\rho^n &= \frac{\tilde\rho^n}{\|\tilde\rho_n\|_{L^1(\m)}}\log\left(\frac{\tilde\rho^n}{\|\tilde\rho_n\|_{L^1(\m)}}\right)\\
&= \frac{\tilde\rho^n}{\|\tilde\rho_n\|_{L^1(\m)}}
\log\frac{\tilde\rho^n}{\|\tilde\rho_n\|_{L^1(\m)}}\chi_{\{\tilde \rho^n\geq1/e\}}+\frac{\tilde \rho^n}{\|\tilde\rho_n\|_{L^1(\m)}}\log\frac{\tilde \rho^n}{\|\tilde\rho_n\|_{L^1(\m)}}\chi_{\{\tilde\rho^n<1/e\}}\\
&\leq \frac{\rho}{1-\epsilon}
\log\frac{\rho}{1-\epsilon}\chi_{\{\rho_0\geq1/e\}}+(\rho-\epsilon)\log(\rho-\epsilon)\chi_{\{\rho<1/e\}}.
\end{align*}
As $x\mapsto x\log x$ is smooth, we find that there is some universal constant $C\in\R$ such that for all sufficiently small $\epsilon$, we have $\frac{\rho}{1-\epsilon}\log\frac{\rho}{1-\epsilon}\leq \rho\log\rho +C\epsilon$ and $(\rho-\epsilon)\log(\rho-\epsilon)\leq \rho\log\rho+C\epsilon$. Then we obtain 
\[\rho^n\log\rho^n\leq \rho\log\rho+C\epsilon,\]
for all sufficiently small $\epsilon$ and all sufficiently large $n$. By MCT, we see that $\Ent(\mu^n)\to\Ent(\mu)$, so $U_N(\mu^n)\to U_N(\mu)$. A similar proof shows that $S_N(\mu^n)\to S_N(\mu)$.

Therefore it would be the case that $\{\tilde\mu^n\}$ is a simple sequence for $\mu$, except that the sets $\tilde A_j^n$ are not compact. To remedy this, we use regularity of $\m$. For each $n\in\N$, and each $j=1,\dots, M_n$, there is compact set $A_j^n\subset \tilde A_j^n$ such that $\m(\tilde A_j^n\setminus A_j^n)<\m(\tilde A_j^n)/n$. 

We then define 
\[\mu^n:=\sum_{j=1}^{M_n}\lambda_j^n\m_{A_j^n}.\]
We observe that 
\[\m_{\tilde A_j^n}-\frac1n\leq\m_{A_j^n} \leq \frac{n}{n-1}\m_{\tilde A_j^n},\]
which implies that 
\begin{align}\label{tilderhoestimate}
\tilde \rho^n-\frac 1n\leq \rho^n\leq \frac{n}{n-1}\tilde\rho^n,
\end{align}
where $\rho^n:=\diff{\mu^n}{\m}$ and $\tilde\rho^n:=\diff{\tilde\mu^n}{\m}$. Equation \eqref{tilderhoestimate} implies that $\mu^n$ satisfies all the above properties, as well as the compactness of the supporting sets $A_j^n$. Thus $\{\mu^n\}_{n\in \N}$ is a simple sequence. 
\end{proof}

Next is a technical lemma which allows us to approximate midpoints between measures with quantitatively small support. 
 
\begin{lemma}[Midpoint approximation lemma]\label{midpointapproxlemma} Let $(X,\ell, \tau,\m)$ be a compact measured synthetic spacetime, and suppose $\ell_+$ is continuous. Then for each metrizer $d$ of $X$ and each $\epsilon>0$, there exists $\delta>0$ such that for each pair of totally timelike measures $(\mu_0,\mu_1)\in\Pcal(X)^2$ and each $\mu_t\in \Pcal(X)$ such that $\spt\mu_t\subset G_t(\spt\mu_0,\spt\mu_1)$, if $\diam_d(\spt\mu_0),\diam_d(\spt\mu_1)<\delta$ then we have 
\begin{align}
\ell_q^q(\mu_0, \mu_t)\geq t^q\ell^q_q(\mu_0,\mu_1)-\epsilon;\quad 
\ell_q^q(\mu_t, \mu_1)\geq (1-t)^q\ell^q_q(\mu_0,\mu_1)-\epsilon.
\end{align}
\end{lemma}
\begin{proof}
Let $d$ be a metrizer of $X$ and $\epsilon>0$. Let $(\mu_0,\mu_1)\in\Pcal(X)^2$ be $q$-timelike dualizable and let $\mu_t\in\Pcal(X)$ be such that $\spt\mu_t\subset G_t(\spt\mu_0,\spt\mu_1)$. By compactness of $X$, the family of maps $\Big\{x\mapsto \ell_+(x,y)\Big\}_{y\in X}$ is $d$-uniformly equicontinuous. Therefore, given $\epsilon>0$ there exists $\delta_0>0$ such that for all $x,x',y\in X$, $d(x,x')<\delta_0\implies |\ell_+^q(x, y)-\ell_+^q(x',y)|<\epsilon/2$. 
A similar argument gives the existence of a $\delta'>0$ such that for all $(x,y),(x',y')\in X^2$, 
\[\max\{d(x,x'),d(y,y')\}<\delta'\implies |\ell_+^q(x,y)-\ell_+^q(x'y,')|<\epsilon/2\]

 Suppose $\diam_d(\spt\mu_0)<\min\{\delta_0,\delta'\}$, and $\diam_d(\spt\mu_1)<\delta'$. For any fixed $x'\in \spt\mu_0$, and each $y\in \spt\mu_t\subset G_t(\spt\mu_0, \spt\mu_1)$ there is $z(y)\in \spt\mu_1$ such that $\ell(x', y)=t\ell(x', z(y))$. Observe then that for any $(x',z')\in \spt\mu_0\times \spt\mu_1$: 
\begin{align*}
\ell_q^q(\mu_0, \mu_t) &\geq \int\ell^q(x,y)d\mu_0(x)d\mu_t(y) = \int \ell_+^q(x,y)d\mu_0(x)d\mu_t(y)\geq \int \left(\ell_+^q(x',y)-\frac{\epsilon}{4}\right) d \mu_0(x)d\mu_t(y)\\
&= \int \ell_+^q(x',y)d\mu_t(y)+\frac\epsilon4 =  t^q\int \ell_+^q(x',z(y))d\mu_t(y)-\frac\epsilon2\geq t^q\int \ell_+^q(x',z')d\mu_t(y)-\frac\epsilon2 \\
&=  t^q\int \left(\ell_+^q(x',z')-\frac\epsilon2\right)d\mu_t(y)-\frac\epsilon2= t^q\ell_+^q(x',z')-\epsilon.
\end{align*}

As $(\mu_0,\mu_1)$ are totally timelike, there is a timelike optimal coupling $\pi\in \TOpt_q(\mu_0, \mu_1)$. Integrating the above inequality with respect to $\pi$ yields: 
\[\ell_q^q(\mu_0, \mu_t) = \int \ell_q^q(\mu_0, \mu_1)d\pi(x',z')\geq t^q\int \ell_+^q(x',z')d\pi(x',z')-\epsilon = t^q\int \ell^q(x',z')d\pi(x',z')-\epsilon=t^q\ell_q^q(\mu_0, \mu_1)-\epsilon. \]

 An exactly symmetric argument gives the existence of a $\delta_1>0$ such that if $\diam_d(\spt\mu_1)<\min\{\delta',\delta_1\}$ and $\diam_d(\spt\mu_0)<\delta'$, then 
\[\ell_q^q( \mu_t,\mu_{1})\geq (1-t)^q\ell_q^q( \mu_0,\mu_1)-\epsilon.\]

Then taking $\delta=\min\{\delta_0,\delta',\delta_1\}$ gives the result. 
\end{proof} 

\section{Main Results}

\subsection{Curvature-Dimension Spacetimes are Strong Brunn-Minkowski Spacetimes}
In this section we prove that timelike curvature dimension spaces admit a strong Brunn-Minkowski inequality. 
\begin{theorem} Let $N>1$, $K\in\R$, and $q\in (0,1)$. If $X\in \TCD_q(K,N)$, then $X\in \sTBM_q(K,N')$ for all $N'>N$, and if  $X\in \TCD_q^e(K,N)$, then $X\in \sTBM_q^*(K,N)$. 
\end{theorem}
\begin{proof}
Suppose $X\in \TCD_q(K,N)$. Let $(A,B)$ be a pair of compact $q$-timelike sets. Then there exists a geodesic $\mu_t$, and a timelike $q$-optimal plan $\pi\in \TOpt_q(\m_A,\m_B)$, such that for all $N'>N$, 
\[S_{N'}(\mu_t) \leq -\int \tau_{K,N'}^{(1-t)}(\ell(x,y))\m^{1/N'}(A)+ \tau_{K,N'}^{(t)}(\ell(x,y))\m^{1/N'}(B)d\pi(x,y) \]

Let $\Lambda= \left(\int \ell^2d\pi\right)^{1/2}$. By the convexity of $\theta\mapsto \tau_{K,N'}^{(t)}(\sqrt{\theta})$, we have, by Jensen's inequality, that 
\[\int \tau_{K,N'}^{(1-t)}(\ell(x,y))\m^{1/N'}(A)d\pi(x,y) \geq  \m^{1/N'}(A) \tau_{K,N'}^{(1-t)}\left(\int \ell^2(x,y) d\pi(x,y)\right)^{1/2} = \tau_{K,N'}^{(1-t)}(\Lambda)\m^{1/N'}(A). \]

Similarly, we have 
\[\int \tau_{K,N'}^{(t)}(\ell(x,y))\m^{1/N'}(B)d\pi(x,y) \geq   \tau_{K,N'}^{(t)}(\Lambda)\m^{1/N'}(B).\]

Jensen's inequality also gives that $-\m^{1/N'}(\spt\mu_t) \leq S_{N'}(\mu_t)$. Combining these inequalities gives 
\[\m^{1/N'}(\spt\mu_t) \geq \tau_{K,N'}^{(1-t)}(\Lambda)\m^{1/N'}(A)+ \tau_{K,N'}^{(t)}(\Lambda)\m^{1/N'}(B),\]

Monotonicity of $\theta\mapsto \tau_{K,N}^{(t)}(\theta)$ then gives 
\[\m^{1/N'}(\spt\mu_t) \geq \tau_{K,N'}^{(1-t)}(\Theta)\m^{1/N'}(A)+ \tau_{K,N'}^{(t)}(\Theta)\m^{1/N'}(B),\]
as desired.

If $X\in \TCD^e(K,N)$, then, 
\[U_{N}(\mu_t)\geq \sigma_{K/N}^{(1-t)}(\Lambda)\m^{1/N}U_{N}(\m_A)+\sigma_{K/N}^{(t)}(\Lambda)\m^{1/N}U_{N}(\m_B) \]
Using the fact that $U_{N'}(\m_A)= \m^{1/N'}(A)$, $U_{N'}(\m_B)= \m^{1/N'}(B)$, and $U_{N'}(\mu_t)\leq \m^{1/N'}(\spt\mu_t)$, and monotonicity of $\theta\mapsto\sigma_k^{(t)}(\theta)$, we get the desired result. 
\end{proof}

\subsection{Brunn-Minkowski Spacetimes are Measure Contracting Spacetimes} 

In this section we prove that $\TBM^*$ spaces are $\TMCP^e$. This step is needed in order to obtain unique optimal maps, which are necessary to prove all of the following results. Most of the arguments presented here will be recapitulated in the following section.

The following is a well-known lemma which allows us to reduce the $\TMCP$ condition to entropic convexity along $t$-midpoints, rather than displacement geodesics. The proof can be found within \cite[Thm.\ 5.7]{Octet}. We rewrite it explicitly for the sake of completeness. Later, we will need a similar version of this result for the $\TCD_q^e(K,N)$ condition (\Cref{midpointTCDe}), but the careful reader can see that the proof would be essentially the same, so we omit it. A similar proof of this result for $\TCD_q(K,N)$ spaces can be found within \cite[Thm.\ 3.33]{Braun}. 

\begin{lemma}[$\TMCP$ on $t$-midpoints]\label{midpointTMCP}
Let $X$ be a measured globally hyperbolic synthetic spacetime. Suppose there exists $K\in\R$ and $N>1$ such that, for every $\mu=\rho\m\in \Pcal_{c}^{ac}(X)$, every $x_0\in X$ such that $\spt\mu_0\subset I^-(x_0)$, and every $t\in (0,1)$, there exists a $t$-midpoint $\mu_{t}$ between $\mu$ and $\delta_{x_0}$ such that 
\begin{align}\label{midpointtmcp}
U_N(\mu_{t})\geq \sigma_{K/N}^{(1-t)}(\Lambda)U_N(\mu),
\end{align}
where $\Lambda:=\sqrt{\int \ell^2(x,x_0)d\mu(x)}$. Then $X\in \TMCP^e(K,N)$. 
\end{lemma}
\begin{proof}
The aim is to show that we can construct a geodesic joining $\mu$ to $\delta_{x_0}$ which satisfies the entropic convexity.  

Fix $\mu\in \Pcal^{ac}_c(X)$ and $x_0$ such that $\spt\mu\subset I^-(x_0)$. Let $\Ical_t(\mu, \delta_{x_0})$ denote the set of $t$-midpoints from $\mu$ to $\delta_{x_0}$ which satisfy \eqref{midpointtmcp}. Let $Y=J(\spt\mu, x_0)$, which is a compact synthetic spacetime by global hyperbolicity and \Cref{spacetimestructure}. Notice that if $\nu\in \Pcal(Y)$, then $\Ical_t(\nu,\delta_{x_0})\subset \Pcal(Y)$. 

For each $t\in (0,1)$, iteratively define the collection of measures $\{\mu^m_t\}_{m=1}^\infty$ such that $\mu^1_t\in \Ical_t(\mu, \delta_{x_0})$, and $\mu^m_t\in \Ical_t(\mu_{t}^{m-1}, \delta_{x_0})$. Notice then that $\mu^m_t\in \Ical_{1-(1-t)^{m-m'}}(\mu^{m'}_t, \delta_{x_0})$, for any $m>m'$. Letting $\{t_n\}_{n\in \N}$ be a sequence decreasing to $0$, we arrive at a countable collection of measures $\{\mu^m_n:= \mu^m_{t_n}\}_{n,m\in \N}$. 

Compactness of $Y$ implies narrow compactness of $\Pcal(Y)$ and thus $\{\mu^m_n\}_{n,m}$ is narrowly precompact. Let $D\subset (0,1)$ denote a countable dense subset of $(0,1)$. For each $s\in D$, there is a sequence $\{m_{n,s}\}_{n\in\N}\subset \N$ such that $1-(1-t_n)^{m_{n,s}}\to s$ as $n\to \infty$. Narrow compactness of $\Pcal(Y)$ together with a standard diagonalization argument shows that the countable collection $\{\mu^n_s:=\mu^{m_{n,s}}_n\}_{n\in\N}$ admits a subsequential limit $\mu_{s}$ for all $s\in D$. Notice that for each $s\in D$, $\mu_s$ satisfies \eqref{midpointtmcp} (with $t=s$), by narrow lower semicontinuity of $\Ent$ and continuity of the right hand side in $t$. 

As $\Pcal(Y)$ is compact we may complete $D\ni s\mapsto \mu_s$ to a causal path on $[0,1]$. It then remains to show that  $s\mapsto \mu_s$ is a geodesic. Indeed, for $s_1\leq s_2\in D$, we have, by upper-semicontinuity of $\ell_q$, 
\begin{align*}
\ell_q(\mu_{s_1},\mu_{s_2})&\geq \limsup_{n\to\infty} \ell_q(\mu^n_{s_1}, \mu^n_{s_2})= \limsup_{n\to\infty}1-(1-\lambda_n)^{m_{n,s_2}-m_{n,s_1}}\ell(\mu_s^m,\delta_{x_0})\\
&=\limsup_{n\to\infty}\Big(1-(1-\lambda_n)^{m_{n,s_2}-m_{n,s_1}}\Big)(1-\lambda_n)^{m_{n,s_1}}\ell_q(\mu, \delta_{x_0})\\
&=(s_2-s_1) \ell_q(\mu, \delta_{x_0}),
\end{align*}
which then proves by upper-semicontinuity that $\ell_q(\mu_{s_2},\mu_{s_1})\geq (s_2-s_1)\ell_q(\mu, \delta_{x_0})$ for all $s_1,s_2\in [0,1]$ with $s_1\leq s_2$. Thus $s\mapsto \mu_s$ is an $\ell_q$-path. Since $Y$ is regular, $\Pcal(Y)$ is regular by \Cref{spacetimestructure}, and so $s\mapsto \mu_s$ is a geodesic, by \Cref{geodesicpathtocurve}. 
\end{proof}

Next we obtain a timelike Brunn-Minkowski inequality from sets to points. The proof is a simple approximation argument. 
\begin{lemma}[$\TBM(K,N)$ for Dirac targets]\label{DiracTBM} Let $(X,\ell,\tau,\m)$ be a measured synthetic spacetime, with $\ell_+$ continuous. Suppose $X\in \TBM^*(K,N)$. Let $x_0\in X$ and let $A\subset X$ be a compact set of positive measure such that $A\subset I^-(x_0)$. Then
\[\m^{1/N}(G_t(A,x_0))\geq \sigma_{K/N}^{(1-t)}(\Theta(A,x_0))\m^{1/N}(A).\]
\end{lemma}
\begin{proof}
Let $d$ be a metrizer of $X$. Then the collection $\{U_n:= B^d(x_0, 1/n)\}_{n\in \N}$ is a countable local base for $x_0$. Since $\m$ is Radon with full support, $\m(U_n)>0$, and there exists a compact set $K_n\subset U_n$ such that $\m(K_n)>0$. The set $B_n:=K_n\cup\{x_0\}$ is then compact, contains $x_0$, and is contained in $U_n$. It follows that $\{B_n\}$ is a compact local base for $x_0$ of positive $\m$-measure, and WLOG we may assume the sequence is nesting. 

As $\ell_+$ is continuous, $A$ is compact, and $(A,x_0)$ is $q$-timelike, the map $x\mapsto \inf_{y\in A}\ell_+(y,x)$ is well-defined, lower-semicontinuous and positive at $x=x_0$. Therefore there is a neighbourhood $U$ of $x_0$ where $\ell_+(y,x)>0$ for all $(y,x)\in A\times U$, so the pair $(A,U)$ is $q$-timelike. Therefore for all sufficiently large $n$, the pair $(A,B_n)$ is totally timelike. 

Let $A_t^n:=G_t(A,B_n)$, and $A_t:=G_t(A,x_0)$. From the $\TBM^*(K,N)$ property, we have 
\[\m^{1/N}(A_t^n)\geq \sigma_{K/N}^{(1-t)}(\Theta(A,B_n))\m^{1/N}(A)+\sigma_{K/N}^{(t)}(\Theta(A,B_n))\m^{1/N'}(B_n) \]

As the $B_n$ are nested, the $A_t^n$ are nested, and $A_t\subset \bigcap_{n=1}^\infty A_t^n$. We claim that $A_t:= \bigcap_{n=1}^\infty A_t^n$. Indeed, let $z\in \bigcap_{n=1}^\infty A_t^n$. Then for each $n\in \N$, there is $x_n\in A$ and $y_n\in B_n$ such that $z$ is a $t$-midpoint between $x_n$ and $y_n$. By construction it is clear that $y_n\to x_0$. Compactness of $A$ implies that a subsequence of $x_n$ converges to some point $x\in A$. Passing to subsequences, we see that $\ell_+(x_n,z)\to \ell_+(x,z)$, whereas $\ell_+(x_n,y_n)\to \ell(x,x_0)$. It follows that $\ell_+(x,z)=t\ell_+(x,x_0)$ and similarly $\ell(z,x_0)=(1-t)\ell_+(x,x_0)$. Therefore $z$ is a $t$-midpoint between $x\in A$ and $x_0$, so $z\in A_t$. This proves the claim. 

Moreover continuity of $\ell_+$ over the compact sets $A\times B_1$ implies that for each $\epsilon>0$, there is a $\delta>0$ such that for all $x_1,x_2\in X$, 
\[d(x_1,x_2)<\delta\implies |\ell_+(y,x_1)-\ell+(y,x_2)|<\epsilon.\]
It follows that there is $n_0\in \N$ such that for all $n>n_0$ and all $x_1,x_2\in B_n$, we have 
\[|\ell_+(y,x_1)-\ell_+(y,x_2)|<\epsilon\]
and so we have for $K\geq 0$
\[\liminf_{n\to\infty} \Theta(A,B_n)\geq \Theta(A,x_0),\]
and for $K<0$, 
\[\limsup_{n\to\infty} \Theta(A,B_n)\leq \Theta(A,x_0).\]
In either case, continuity of $\theta\mapsto \sigma_{k}^{(t)}(\theta)$ uniformly in $t$ then gives
\[\liminf_{n\to\infty} \sigma_{K/N}^{(1-t)}(\Theta(A,B_n))\geq \sigma_{K/N}^{(1-t)}(\Theta(A,x_0)).\]

Finally the fact that $A_t=\bigcap_{n=1}^\infty A_t^n$ and continuity of $\m$ from above then gives,
\begin{align*}
\m^{1/N}(A_t)&=\liminf_{n\to\infty} \m^{1/N}(A_t^n) \geq \liminf_{n\to\infty}\Big( \sigma_{K/N}^{(1-t)}(\Theta(A,B_n))\m^{1/N}(A)+ \sigma_{K/N}^{(t)}(\Theta(A,B_n))\m^{1/N'}(B_n)\Big)\\
&\geq \sigma_{K/N}^{(1-t)}(\Theta(A,x_0))\m^{1/N}(A)+ \sigma_{K/N}^{(t)}(\Theta(A,x_0))\m^{1/N}(x_0)\\
&\geq \sigma_{K/N}^{(1-t)}(\Theta(A,x_0))\m^{1/N}(A),
\end{align*}
as desired.
\end{proof}

We now prove the main result of this section. 
\begin{theorem}[$\TBM^*(K,N)$ spaces are $\TMCP^e(K,N)$ spaces]\label{TBMtoTMCP} Fix $0<q<1$. Let $(X,\ell,\tau,\m)$ be a timelike $q$-essentially non-branching measured globally hyperbolic synthetic spacetime, with $\ell_+$ continuous. If $X\in \TBM^*(K,N)$, then $X\in \TMCP^e(K,N)$. 
\end{theorem}
\begin{proof}
Let $\mu\in\Pcal_c^{ac}(X)$, and $x_0\in X$ be such that $(\spt\mu,x_0)$ is timelike. In light of \Cref{midpointTMCP}, we aim to show that for each $t\in(0,1)$, there is a $t$-midpoint $\mu_t$ joining $\mu$ to $\delta_{x_0}$ such that 
\begin{align}\label{convexityinequality1}
U_N(\mu_t)\geq\sigma_{K/N}^{(1-t)}(\Lambda)U_N(\mu),
\end{align}
where $\Lambda:= \sqrt{\int \ell^2(x,x_0)d\mu(x)}$. Fix $t\in (0,1)$.

\paragraph{Step 1: Obtaining the $t$-midpoint $\mu_t$.}

From global hyperbolicity, the set $Y:=J(\spt\mu, x_0)$ is a compact synthetic spacetime. For each $n\in\N$, we apply \Cref{midpointapproxlemma} with the choice of metrizer $d$ of $Y$ and $\epsilon=1/n$ to obtain some $\delta_n>0$ such that for any pair $(\mu_0, \mu_1)\in \Pcal(Y)^2$ of compactly supported totally timelike measures, and any measure $\mu_t\in \Pcal(Y)$ with $\spt\mu_t\subset G_t(\mu_0, \mu_1)$,  if $\diam_d(\spt\mu_i)<\delta_n$ for $i=0,1$, then 
\[\ell_q(\mu_0, \mu_t) \geq t\ell_q(\mu_0, \mu_1)-\frac 1n; \quad \ell_q(\mu_t, \mu_1) \geq (1-t)\ell_q(\mu_0, \mu_1)-\frac1n.\]

WLOG we may assume $\delta_n\to 0$ as $n\to\infty$. By \Cref{simplesequence} there exists a simple sequence $\{\mu^n\}_{n\in\N}$ which converges to $\mu$. From the controlled decomposition property, we may write $\mu^n:=\sum_{j=1}^{M_n}\lambda_j^n \m_{A_j^n}\in \Pcal(Y)$, where for each $n$, the collection $\{A_j^n\}_{j=1}^{M_n}$ are compact, mutually disjoint, have finite positive $\m$-measure, $A_j^n\subset \spt\mu\subset Y$, and $\diam_d(A_j^n)<\delta_n$. 

Let $A_{t,j}^n:=G_t(A_j^n,x_0)\subset Y$. From \Cref{DiracTBM}, we know that $A_{t,j}^n$ has positive $\m$-measure. Since $A_{t,j}^n$ is compact and $\m$ is Radon, we also know that it has finite $\m$-measure. We define 
\[\mu_{t}^n:=\sum_{j=1}^{M_n}\lambda_j^n \m_{A_{t,j}^n}.\]

Observe that $\mu_t^n\in \Pcal(Y)$, and so $\mu_t^n$ admits a narrow subsequential limit $\mu_t$. Concavity of $\ell_q^q$ and \Cref{midpointapproxlemma} applied to the triple $(\mu_0,\mu_t,\mu_1):= (\m_{A_j^n}, \m_{A_{t,j}^n}, \delta_{x_0})$ then gives that
\begin{align*}
\ell_q^q(\mu^n,\mu_t^n) \geq \sum_{j=1}^{M_n}\lambda_j^n \ell_q^q(\m_{A_j^n}, \m_{A_{t,j}^n})
&\geq \sum_{j=1}^{M_n} \lambda_j^n t^q\ell_q^q(\m_{A_j^n}, \delta_{x_0})-\frac1n = t^q\sum_{j=1}^{M_n}\lambda_j^n \int\ell^q(x,x_0)d\m_{A_j^n}(x)-\frac 1n \\
& = t^q\int\ell^q(x,x_0) d\mu^n(x)-\frac 1n 
\end{align*}
Narrow convergence of $\mu^n\narrowto\mu$, continuity of $\ell_+^q$, and positivity of $\ell^q$ on $\spt\mu\times \{x_0\}$, and upper semicontinuity of $\ell_q$ then gives that, after passing to subsequence, that for all $\epsilon>0$,
\begin{align*}
\ell_q^q(\mu, \mu_t) \geq \limsup_{n\to\infty} \ell_q^q(\mu^n, \mu_t^n) = t^q\limsup_{n\to\infty}\int \ell^q(x,x_0) d\mu^n(x)-\frac 1n=t^q\int \ell^q(x,x_0)d\mu(x).
\end{align*}
Hence we may conclude that 
\[\ell_q(\mu,\mu_t)\geq t\ell_q(\mu,\delta_{x_0}).\]
A symmetric argument gives also that 
\[\ell_q(\mu_t,\delta_{x_0})\geq (1-t)\ell_q(\mu,\delta_{x_0}),\]
which proves that $\mu_t$ is indeed a $t$-midpoint between $\mu$ and $\delta_{x_0}$. \\

\paragraph{Step 2: Entropic Convexity along $\mu_t$.}

We now verify the entropic convexity inequality \eqref{convexityinequality1} along $\mu_t$. The $q$-essential non-branching assumption guarantees that for each $n$, the sets $A_{t,j}^n$ are mutually disjoint for all $j=1,\dots M_n$, since the $A_j^n$ are mutually disjoint. As a result, we have 

\begin{align*}
\Ent(\mu_t^n) &= \sum_{j=1}^{M_n}\Ent(\lambda_j^n\m_{A_{t,j}^n}) = \sum_{j=1}^{M_n} \lambda_j^n\log\frac{\lambda_j^n}{\m(A_{t,j}^n)}.
\end{align*}
Therefore we have 
\begin{align*}
U_N(\mu_t^n) &=\exp\left(-\frac1N\Ent(\mu_t^n)\right) =\prod_{j=1}^{M_n} \left(\frac{\m(A_{t,j}^n)}{\lambda_j^n}\right)^{\lambda_j^n/N} \geq \prod_{j=1}^{M_n}\left(\sigma_{K/N}^{(1-t)}(\Theta_j^n)\frac{\m^{1/N}(A_{0,j}^n)}{(\lambda_j^n)^{1/N}}\right)^{\lambda_j^n } \\
&= \prod_{j=1}^{M_n}\left(\sigma_{K/N}^{(1-t)}(\Theta_j^n)\right)^{\lambda_j^n} U_N(\mu^n),
\end{align*}
where $\Theta_j^n=\Theta(A_j^n,x_0)$. Now observe that log-convexity of $\theta\mapsto \sigma_{k}^{(t)}(\sqrt \theta)$ implies that 
\begin{align}\label{distortionestimate1}
\prod_{j=1}^{M_n} \sigma_{K/N}^{(1-t)}(\Theta_j^n)^{\lambda_j^n}\geq \sigma_{K/N}^{(1-t)}\left(\Big(\sum_{j=1}^{M_n} \lambda_j^n (\Theta_j^n)^2\Big)^{1/2}\right).
\end{align}
Continuity of $\ell$, compactness of $A_j^n$, and the fact that $\diam_d(A_j^n)<\delta_n$ gives that for each there exists $n_0\in\N$ such that for all $n>n_0$, and all $x_j^n\in A_j^n$, 
\[|\Theta_j^n-\ell(x_j^n,x_0)|<1/n.\]
It follows that, with $\Lambda^n:= \sqrt{\int \ell^2(x,x_0)d\mu^n}$, we have 
\[\left| \Big(\sum_{j=1}^{M_n} \lambda_j^n (\Theta_j^n)^2\Big)^{1/2}-\Lambda^n\right|<\frac 1n,\]
And so by monotonicity of $\theta\mapsto \sigma_{K/N}^{(t)}(\theta)$ we obtain 
\[U_N(\mu_t^n) \geq \sigma_{K/N}^{(1-t)}\left(\Lambda^n\mp \frac 1n\right)U_N(\mu^n), \]
where we choose $-\frac 1n$ if $K\geq 0$ and $+\frac 1n$ if $K<0$. Narrow convergence of $\mu^n$ to $\mu$ guarantees that $\Lambda^n\to \Lambda$, and the continuity along entropy guarantees that $U_N(\mu^n)\to U_n(\mu)$. Thus from upper semicontinuity we obtain
\[U_N(\mu_t)\geq \limsup_{n\to\infty} U_N(\mu_t^n) \geq \limsup_{n\to\infty}\sigma_{K/N}^{(1-t)}\left(\Lambda^n\mp \frac1n\right)U_N(\mu^n)= \sigma_{K/N}^{(1-t)}(\Lambda)U_N(\mu).\]
The last equality holds since $\sigma_{K/N}^{(1-t)}(\Lambda)>0$. 

\end{proof}

\subsection{Equivalence between $\sTBM_q^*(K,N)$ and $\TCD_q^e(K,N)$}
The main result of this section is to prove the equivalence between $\sTBM_q^*(K,N)$ and the $\TCD_q^e(K,N)$.  We hope to reproduce the main elements of the proof of \Cref{TBMtoTMCP}. Namely, we require the following elements: 
\begin{enumerate}
\item A sequence of pairs $(\mu_0^n,\mu_1^n)$ of simple measures which converges to the target pair $(\mu_0,\mu_1)$;
\item A sequence $(\mu_t^n)$ of simple measures that a) converges to a $t$-midpoint $\mu_t$ between $\mu_0$ and $\mu_1$, and b) attains the correct estimate in the entropy $\Ent$. 
\end{enumerate}

The following will help establish the correct sequence of pairs of simple measures, which we call a ``good simple sequence''.

\begin{prop}[Existence and uniqueness of optimal maps, \text{\cite[Thm.\ 5.12]{Octet}}]\label{optimalmaps}
Suppose $(X,\ell, \tau, \m)$ is a $q$-essentially non-branching measured globally hyperbolic synthetic spacetime, with $\ell_+$ continuous, and suppose $X\in \TMCP^e(K,N)$. Then $X$ for every pair $(\mu_0,\mu_1)\in \Pcal^{ac}(X)\times \Pcal(X)$ of $q$-separated measures (see \Cref{chronologynotions}), there exists a $\mu_0$-essentially unique Borel map $T:X\to X$ such that $\Opt_q(\mu_0,\mu_1)=\{(\id\times T)_\# \mu_0\}$. Such a map $T$ is called an {\it $q$-optimal map}.
\end{prop}

\begin{lemma}[Existence of good simple sequences]\label{goodsimplesequence}
Let $X$ be a measured globally hyperbolic synthetic spacetime, and moreover suppose that for each totally timelike pair $(\mu_0,\mu_1)\in \Pcal^{ac}_c(X)^2$ admits a $q$-optimal map. Then for each totally timelike pair $(\mu_0,\mu_1)\in \Pcal_c^{ac}(X)^2$, there exists a sequence of pairs of simple measures $(\mu_0^n, \mu_1^n)\in\Pcal_c^{ac}(X)^2$, which we call a ``good simple sequence'', which satisfies the following properties: 

\begin{enumerate}
\item (Entry-wise simplicity) For each $i=0,1$, $\{\mu_i^n\}_{n\in\N}$ is a simple sequence (see \Cref{simplesequence}) converging to $\mu_i$.
\item (Correlated controlled decomposition) For each metrizer $d$ of $X$, and each positive sequence $\delta^n\to 0$, we may decompose $\mu_i^n$ as 
\[\mu_i^n= \sum_{j=1}^{M_n}\lambda_j^n \m_{A_{ij}^n},\]
where, for each $i=0,1$, the collection $\{A_{ij}^n\}_{j=1}^{M_n}$ consists of compact mutually disjoint sets of positive and finite $\m$-measure, $A_{ij}^n\subset \spt\mu_i$, and $\diam_d(A_{ij}^n)<\delta^n$.  
\item (Stability of couplings) Given the decomposition as above, and given $\pi_j^n\in \Opt_q(\m_{A_{0j}^n}, \m_{A_{1j}^n})$, the coupling $\pi^n:=\sum_{j=1}^{M_n}\lambda_j^n \pi_j^n$ is in $\Opt_q(\mu_0^n,\mu_1^n)$, and admits a narrow subsequential limit $\pi\in \TOpt_q(\mu_0, \mu_1)$. 
\end{enumerate}
\end{lemma}

\begin{proof}
Let $(\mu_0,\mu_1)\in\Pcal_c^{ac}(X)^2$ be a totally timelike pair, so $\spt\mu_0\times \spt\mu_1\subset X_\ll^2$. From \Cref{simplesequence}, we have the existence of simple sequences $\tilde\mu_i^n$ which narrowly converge to $\mu_i$, given by the controlled decomposition 
\[\tilde \mu_i^n := \sum_{j=1}^{M_{i,n}} \alpha_{ij}^n \m_{B_{ij}^n},\]
which depends on the metrizer $d$ and the positive sequence $\delta^n\to 0$. 

Recall that by assumption, for each $n\in\N$, there is a $q$-optimal map $T_n$ joining $\tilde\mu^n_0$ to $\tilde\mu_1^n$. Likewise, there is a $q$-optimal map $T_n^{-1}$ joining $\tilde\mu^n_1$ to $\tilde\mu_0^n$, and uniqueness implies that $T_n^{-1}\circ T_n(x)=x$ for $\tilde\mu_0$-a.e $x\in X$. From Lusin's theorem, there is compact set $K_n$ such that $\tilde\mu_0^n( K_n)>1-\epsilon$ and $T_n|_{K_n}$ is continuous. There is also compact $K_n^{-1}$ such that $\tilde \mu_1^n(K_n^{-1})>1-\epsilon$ and $T_n^{-1}|_{K_n^{-1}}$ is continuous. Since $\tilde \mu_1^n=(T_n)_\# \tilde \mu_0^n$, WLOG we may assume that $K_n^{-1}=T(K_n)$.  

We define $ A_{0,jk}^n:= B_{0,j}^n \cap K_n\cap T_n^{-1}(B_{1,k}^n\cap K_n^{-1})$, and $ A_{1,jk}^n:= T_n( A_{0,jk}^n)\cap K_n^{-1}\cap B_{1,k}^n$. Continuity of the maps $T_n$ and $T_n^{-1}$ ensures that the sets $A_{ijk}^n$ are compact. In fact one observes that the collection $\{ A_{i,jk}^n\}_{ijkn}$ is still a controlled decomposition, in the sense of \Cref{simplesequence}; that is, for each $n\in\N$, and each $i,j,k$, the set $A_{ijk}^n$ is compact, disjoint from $A_{i'j'k'}^n$ whenever $(i,j,k)\neq (i',j',k')$, and $\diam_d(A_{ijk}^n)<\delta^n$. Moreover 
\[\tilde \mu_1^n( A_{1,jk}^n)=(T_n)_\#\tilde \mu_0^n( A_{1,jk}^n) =\tilde\mu_0^n(T_n^{-1}(A_{1,jk}^n))=\tilde\mu_0^n(A_{0,jk}^n\cap T^{-1}_n(K_{1,k}^n\cap K_n^{-1})) = \tilde\mu_0(A_{0,jk}^n).  \]

After throwing away sets of null $\tilde\mu_0^n$ -measure and renumerating, we define $\{A_{i,j}^n\}_{j=1}^{M_n}:=\{A_{i,jk}^n\}_{j,k}$, and define 
\[ \mu_i^n:= \sum_{j=1}^{M_n}\lambda_{j}^n \m_{A_{i,j}^n},\]
where 
\[\lambda_{j}^n:= \frac{\tilde\mu_{0}^n( A_{0,j}^n)}{\sum_{j=1}^{M_n}\tilde\mu_{0}^n(A_{0,j}^n)}=\frac{\tilde\mu_{1}^n( A_{1,j}^n)}{\sum_{j=1}^{M_n}\tilde\mu_{1}^n( A_{1,j}^n)}.\]
We claim that the pair $(\mu_0^n,\mu_1^n)$ is a good simple sequence for $(\mu_0, \mu_1)$. Recall that to justify such a claim, we must verify 3 properties: 
\begin{enumerate}
\item Entrywise simplicity;
\item Correlated controlled decomposition;
\item Stability of couplings.
\end{enumerate}

Correlated controlled decomposition is satisfied, by construction. Let us verify entrywise simplicity. Recall that $\tilde\mu_i^n$ is a simple sequence converging to $\mu_i$. We need only show that $(\mu_0^n,\mu_1^n)$ is ``close'' to $(\tilde \mu_0^n,\tilde\mu_1^n)$. In particular, observe that since $\tilde\mu_0^n(K_n)>1-\epsilon$, and $\tilde \mu_1^n(K_n^{-1})>1-\epsilon$, for each Borel set $E\subset X$, 
\[\tilde\mu_i^n(E)-\epsilon\leq\mu_i^n(E)\leq (1-\epsilon)^{-1} \tilde\mu_i^n(E)\quad i=0,1.\]

Defining the densities $\rho_i^n$ and $\tilde \rho_i^n$ via $d\mu_i^n:=\rho_i^nd\m$ and $d\tilde\mu_i^n:=\tilde\rho_i^nd\m$, we then see that 
\[\tilde \rho_i^n-\epsilon\leq \rho_i^n\leq \frac{\tilde\rho_i^n}{1-\epsilon}.\]

From this one can immediately check that $\mu_i^n$ is indeed a simple sequence converging to $\mu_i$.

Finally let us prove stability of couplings. Observe that 
\[\mu_0^n= \tilde\mu_0^n(K_n\cap T_n^{-1}(K_n^{-1}))^{-1}\tilde \mu_0^n|_{K_n\cap T_n^{-1}(K_n^{-1})}= \tilde\mu_0^n(K_n)\tilde\mu_0^n|_{K_n},\]
and similarly
\[\mu_1^n= \tilde\mu_1^n(K_n^{-1})\tilde \mu_1^n|_{K_n^{-1}},\]
and therefore it is apparent that $T_n$ is a $q$-optimal map joining $\mu_0^n$ to $\mu_1^n$. 
Moreover observe now that for every Borel set $E\subset X$, we have 
\[\lambda_j^n\m_{\tilde A_{1,j}^n}(E)= \mu_1^n(E\cap \tilde A_{1,j}^n)=\mu_0^n((T_n)^{-1}(E\cap \tilde A_{1,j}^n))=\mu_0(T_n^{-1}(E)\cap \tilde A_{0,j}^n)= \lambda_j^n \m_{\tilde A_{0,j}^n}(T_n^{-1}(E)),\]
which proves that $T_n$ is also $q$-optimal between $\m_{A_{0,j}^n}$ and $\m_{A_{1,j}^n}$. 

This then implies that if $\pi_j^n\in \Opt_q(\m_{A_{0,j}^n},\m_{A_{1,j}^n})$, then $\pi_j^n=(\id\times T_n)_\#\m_{A_{0,j}^n}$, and so 
\[\pi^n:=\sum_{j=1}^{M_n}\lambda_j^n \pi_j^n = \sum_{j=1}^{M_n}\lambda_j^n (\id\times T_n)_\#\m_{A_{0,j}^n}= (\id\times T_n)_\# \mu_0^n \in \Opt_q(\mu_0^n, \mu_1^n).\]
Finally since $\spt\pi^n\subset \spt\mu_0\times \spt\mu_1\subset X_\ll^2$, the sequence $\{\pi^n\}_{n\in\N}$ admits a narrow subsequential limit $\pi\in \TOpt_q(\mu_0,\mu_1)$, by \Cref{Stableopt}.
\end{proof} 

\begin{lemma}[$\TCD_q^e(K,N)$ for $t$-midpoints between totally timelike pairs]\label{midpointTCDe} Fix $0<q<1$, and let $X$ be a timelike $q$-essentially non-branching globally hyperbolic synthetic spacetime, with $\ell_+$ continuous. Suppose for each totally timelike pair $(\mu_0,\mu_1)\in \Pcal_c^{ac}(X)^2$ and each $t\in [0,1]$, there exists a $t$-midpoint $\mu_t$ such that 
\[U_N(\mu_t)\geq \sigma_{K/N}^{(1-t)}(\Theta)U_N(\mu_0)+\sigma_{K/N}^{(t)}(\Theta)U_N(\mu_1),\]
where $\Theta:= \begin{cases}
\inf_{\spt\mu_0\times\spt\mu_1} \ell(x,y) & K\geq 0;\\
\sup_{\spt\mu_0\times \spt\mu_1}\ell(x,y) &K<0.
\end{cases}$. Then $X\in \TCD_q^e(K,N)$.
\end{lemma}
\begin{proof}
We will prove that for each $q$-timelike dualizable pair of measures $(\mu_0,\mu_1)$, and each $t\in [0,1]$, there is a $t$-midpoint $\mu_t$ and a $q$-optimal coupling $\pi\in\TOpt_q(\mu_0,\mu_1)$ such that 
\[U_N(\mu_t)\geq \sigma_{K/N}^{(1-t)}(\Lambda)U_N(\mu_0)+\sigma_{K/N}^{(t)}(\Lambda)U_N(\mu_1),\]
where $\Lambda :=(\int \ell^2d\pi)^{1/2}$. Once proved, a recreation of the proof of \Cref{midpointTMCP} will give show that $X\in \TCD_q^e(K,N)$.

To achieve the aforementioned result, we are inspired by the proof within \cite[Thm. 7.1]{McCann}. 
Let $(\mu_0, \mu_1)\in \Pcal_c^{ac}(X)^2$ be $q$-timelike, and let $\pi\in \Opt_q(\mu_0,\mu_1)$ with $\pi(X_\ll^2)=1$. We may consider $\pi$ as a probability measure on the compact metric space $\spt\mu_0\times \spt\mu_1$, which implies $\pi$ is inner regular. Inner regularity implies that there exists a nested increasing sequence $\{K_n\}_{n\in\N}$ of compact sets contained in $X_\ll^2$ such that for each $\epsilon>0$, there is $n_0$ such that for all $n>n_0$, $\pi(K_n)>1-\epsilon$. Observe that 
\begin{align}\label{pieq1}
\pi|_{K_m}\leq \pi|_{K_n}\leq\pi \quad \forall\ m\leq n.
\end{align}

It is therefore clear that $\pi^n:=\pi(K_n)^{-1}\pi|_{K_n}\to \pi$ as $\epsilon\to 0$, by which we mean the measures converge setwise. Moreover we define $\mu_0^n:= (p_1)_\# \pi^n=\pi(K_n)^{-1} (p_1)_\#\pi|_{K_n}=: \pi(K_n)^{-1} \tilde\mu_0^n$, and so $\mu_0^n\to\mu_0$ as well. We now prove that $\Ent(\mu_0^n)\to \Ent(\mu_0)$. We see from \eqref{pieq1} that $\tilde\mu_0^{m}\leq\tilde\mu_0^n\leq \mu_0 $ for all $m\leq n$. It follows that $\tilde\mu_0^n\ll \m$, and $\tilde \rho_0^n:=\diff{\tilde\mu_0^n}{\m}\leq \rho_0\:=\diff{\mu_0}{\m}$. Clearly then $\mu_0^n\ll \m$, with $\rho_0^n:=\diff{\mu_0^n}{\m}= \pi(K_n)^{-1} \tilde\rho_0^n$. Moreover as $\tilde\mu_0^n\leq \mu_0$ and $\tilde \mu_0^n\to \mu_0$, it's not hard to see that $\tilde \rho_0^n\to \rho_0$ pointwise $\m$-a.e. We then obtain that, for each $\epsilon>0$, there is $n_0$ such that for all $n>n_0$, and $\m$-a.e. $x\in X$, 
\[\rho_0-\epsilon< \rho_0^n\leq (1+\epsilon)\rho_0.\]
Recreating the proof of item {\it 3)} in \Cref{simplesequence} allows us to conclude that $\Ent(\mu_0^n)\to \Ent(\mu_0)$. The same proof also works to show that $\Ent(\mu_1^n)\to \Ent(\mu_1)$, where $\mu_1^n:= (p_2)_\# \pi^n$.

We return now to the main body of the proof. 
Observe that $\spt\pi^n\subset K_n\subset X_\ll^2$. For each $(x,y)\in \spt\pi^n$, there is open set $ A_x^n\times B_y^n$ such that $(x,y)\in A_x\times B_y\subset X_\ll^2$. The set $\{A_x^n \times  B_y^n\}_{(x,y)\in \spt\pi_n}$ then form an open cover of $\spt\pi^n$, and so admit a finite subcover $\{U_k^n\}:=\{A_k^n\times B_k^n \}_{k=1}^{M_n}$. By taking finite intersections, we may assume WLOG that $A_k^n\cap A_{k'}^n=\emptyset$ and $B_k^n\cap B_{k'}^n=\emptyset$ for all $k\neq k'$. 
We then define 
\[\pi^n_k:= \pi^n(U^n_k)^{-1}\pi^n|_{U^n_k}.\]
Then with $\lambda_k^n:= \pi^n(U^n_k)$, we obtain that $\pi^n=\sum_{k=1}^{M_n}\lambda_k^n \pi_k^n$. Note that fixing a metrizer $d$ of $X$, we may also assume WLOG that $\diam_{d^2}(A_k^n\times B_k^n)$ is of sufficiently small diameter such that 
\begin{align}\label{ThetaLambdaestimate}
|(\Theta_k^n)^2-(\Lambda_k^n)^2|< \frac 1n,
\end{align}
where we recall 
\begin{align*}
\Lambda_k^n:= \left(\int \ell^2(x,y) d\pi_k^n\right)^{1/2},\quad
\Theta_k^n:= \begin{cases}
\sup_{A_k^n\times B_k^n}\ell(x,y) & K<0;\\
\inf_{A_k^n\times B_k^n}\ell(x,y) & K\geq 0.
\end{cases}
\end{align*}
Such a sufficiently small diameter exists by the uniform continuity of $\ell$ over the compact set $K\subset X_\ll$.

We observe that, since each $\pi_k^n$ is obtained by restricting the optimal coupling $\pi$, we have that $\pi_k^n\in \Opt_q(\mu_{0k}^n, \mu_{1k}^n)$, where $\mu_{0k}^n:=(p_1)_\# \pi_k^n$ and $\mu_{1k}^n:=(p_2)_\# \pi_k^n$. Moreover since $(\mu_{0k}^n,\mu_{1k}^n)$ are a totally timelike pair, we have, by assumption, the existence of a $t$-midpoint $\mu_{tk}^n$ such that 
\[U_N(\mu_{tk}^n)\geq \sigma_{K/N}^{(1-t)}(\Theta_k^n)U_N(\mu_{0k}^n)+\sigma_{K/N}^{(t)}(\Theta_k^n)U_N(\mu_{1k}^n),\]

We define $\mu_t^n:= \sum_{k=1}^{M_n} \lambda_k^n \mu_{tk}^n$. Compactness of the set $J(\spt\mu_0,\spt\mu_1)$ on which the collection of measures $\{\mu_t^n\}_{n\in \N}$ is uniformly concentrated implies that the collection admits a narrow subsequential limit $\mu_t$. Using the concavity of $\ell_q^q$, and the fact that $\pi^n:= \sum_{k=1}^{M_n}\lambda_k^n \pi_k^n$ is optimal, we have that 
\[\ell_q^q(\mu_0^n, \mu_t^n) \geq \sum_{k=1}^{M_n} \lambda_k^n \ell_q^q(\mu_{0k}^n, \mu_{tk}^n) \geq t^q \sum_{k=1}^{M_n}\lambda_k^n \ell_q^q(\mu_{0k}^n, \mu_{1k}^n)= t^q\ell_q^q(\mu_0^n, \mu_1^n).\]
Upper semicontinuity of $\ell_q$ then guarantees that, after restricting to a subsequence, that $\ell_q(\mu_0,\mu_t)\geq t\ell_q(\mu_0, \mu_1)$, and similarly, $\ell_q(\mu_t,\mu_1)\geq (1-t)\ell_q(\mu_0,\mu_1)$. Therefore $\mu_t$ is a $t$-midpoint joining $\mu_0$ to $\mu_1$. Now from \Cref{singularmidpoint}, we have that for each $i=0,t,1$, $\mu_{ik}^n\perp \mu_{ik'}^n$ for all $k\neq k'$, and so for $i=0,t,1$
\begin{align}\label{linearityofEnt}
\Ent(\mu_i^n) = \sum_{k=1}^{M_\epsilon} \Ent(\lambda_k^n \mu_i^n)=\sum_{k=1}^{M_n} \lambda_k^n \left(\log\lambda_k^n + \Ent(\mu_{ik}^n)\right).
\end{align}
We then obtain, 
\begin{align*}
U_N(\mu_t^n) &= \prod_{k=1}^{M_n} \left(\frac{ U_N(\mu_{tk}^n)}{(\lambda_k^n)^{1/N}}\right)^{\lambda_k^n}\geq \prod_{k=1}^{M_\epsilon} \left(\sigma_{K/N}^{(1-t)}(\Theta_k^n)\frac{U_N(\mu_{0k}^n)}{(\lambda_k^n)^{1/N}}+\sigma_{K/N}^{(t)}(\Theta_k^n)\frac{U_N(\mu_{1k}^n)}{(\lambda_k^n)^{1/N}}\right)^{\lambda_k^n}\\ 
&\overset{(*)}{\geq} \prod_{k=1}^{M_\epsilon} \left(\sigma_{K/N}^{(1-t)}(\Lambda_k^n)\right)^{\lambda_k^n}\left(\frac{ U_N(\mu_{0k}^n)}{(\lambda_k^n)^{1/N}}\right)^{\lambda_k^n}+ \prod_{k=1}^{M_\epsilon} \left(\sigma_{K/N}^{(t)}(\Theta_k^n)\right)^{\lambda_k^n}\left(\frac{ U_N(\mu_{1k}^n)}{(\Theta_k^n)^{1/N}}\right)^{\lambda_k^n}\\ 
&\overset{(**)}{\geq} \sigma_{K/N}^{(1-t)}\left(\left(\sum_{k=1}^{M_n} \lambda_k^n \left(\Theta_k^n\right)^2\right)^{1/2}\right)U_N(\mu_0^n)+\sigma_{K/N}^{(t)}\left(\left(\sum_{k=1}^{M_n} \lambda_k^n \left(\Theta_k^n\right)^2\right)^{1/2}\right)U_N(\mu_1^n)\\
&\overset{(***)}{\geq}\sigma_{K/N}^{(1-t)}\left(\left(\sum_{k=1}^{M_n} \lambda_k^n \left(\Lambda_k^n\right)^2\mp \frac{1}{n}\right)^{1/2}\right)U_N(\mu_0^n)+\sigma_{K/N}^{(t)}\left(\left(\sum_{k=1}^{M_n} \lambda_k^n \left(\Lambda_k^n\right)^2\mp \frac{1}{n}\right)^{1/2}\right)U_N(\mu_1^n)\\
&= \sigma_{K/N}^{(1-t)}\left(\left((\Lambda^n)^2\mp \frac{1}{n}\right)^{1/2}\right)U_N(\mu_0^n)+\sigma_{K/N}^{(t)}\left(\left( (\Lambda^n)^2\mp \frac{1}{n}\right)^{1/2}\right)U_N(\mu_1^n),
\end{align*}
where we use Holder's inequality to obtain the inequality $(*)$, log-convexity of $\theta\mapsto \sigma_{k}^{(t)}(\sqrt\theta)$ together with \eqref{linearityofEnt} to obtain $(**)$, and the monotonicity of $\theta\mapsto \sigma_k^{(t)}(\theta)$, and \eqref{ThetaLambdaestimate} to obtain $(*{*}*)$. In the last inequality, we choose $-\frac1n$ if $K\geq 0$ and $+\frac1n$ if $K<0$. Here $\Lambda^n$ is defined as  
\[\Lambda^n:=\left(\sum_{k=1}^{M_n} \lambda_k^n(\Lambda_k^n)^2\right)^{1/2}=\left(\sum_{k=1}^{M_n}\int \ell^2(x,y) \lambda_k^nd\pi^n_k(x,y)\right)^{1/2}=\left(\int \ell^2(x,y) d\pi^n(x,y)\right)^{1/2}.\] 

Continuity of $\ell_+$ and narrow convergence $\pi^n\narrowto \pi$ gives that $\Lambda^n\to \Lambda$. Finally using the lower semicontinuity of $U_N$, the continuity $\Ent(\mu_i^n)\to \Ent(\mu_i^n)$, and continuity of $\theta\mapsto \sigma_{k}^{(t)}(\theta)$ we obtain that

\[U_N(\mu_0)\geq \sigma_{K/N}^{(1-t)}\left( \Lambda\right)U_N(\mu_0)+\sigma_{K/N}^{(t)}\left(\Lambda\right)U_N(\mu_1),
\]
as desired. Implicitly we have used the fact that $\sigma_{K/N}^{(t)}(\Lambda)>0$. 

\end{proof}
\begin{theorem}\label{sTBMtoTCDe} Fix $0<q<1$. Let $X$ be a timelike $q$-essentially non-branching globally hyperbolic synthetic spacetime with $\ell_+$ continuous. If $X\in \sTBM_q^*(K,N)$, then $X\in \TCD_q^e(K,N)$. 
\end{theorem}
\selectlanguage{english}
\begin{proof}
We mirror the proof of \Cref{TBMtoTMCP}. 
Let $(\mu_0,\mu_1)\in \Pcal_{c}^{ac}(X)^2$ be a totally timelike pair.  From global hyperbolicity, $Y:=J(\spt\mu_0, \spt\mu_1)$ is a compact measured synthetic spacetime. Let $d$ be a metrizer of $Y$. By \Cref{midpointapproxlemma}, for each $n$ there is $\delta_n>0$ such that for each totally timelike pair of measures $(\nu_0, \nu_1)\in\Pcal(Y)^2$ and each $\nu_t\in \Pcal(Y)$ which is concentrated on $G_t(\spt\nu_0, \spt\nu_1)$, if $\max\{\diam_d(\spt\nu_0), \diam_d(\spt\nu_1)\}<\delta_n$, then 
\[\ell_q^q(\nu_0, \nu_t)\geq t^q\ell_q^q(\nu_0, \nu_1)-\frac 1n; \quad \ell_q^q(\nu_t, \nu_1)\geq (1-t)^q\ell_q^q(\nu_0, \nu_1)-\frac 1n.\]

WLOG we may assume that $\delta_n\to 0$. 
From \Cref{goodsimplesequence}, there exists a good simple sequence $(\mu_0^n, \mu_1^n)$ which narrowly converges to $(\mu_0, \mu_1)\in \Pcal(Y)^2$. From the correlated controlled decomposition, we have  $\mu_i^n:= \sum_{j=1}^{M_n} \lambda_j^n\m_{A_{i,j}^n}\in \Pcal(Y)$, and $\diam_d(A_{i,j}^n)<\delta_n$. Applying $\sTBM_q(K,N)$ to the totally timelike pair $(A_{0,j}^n, A_{1,j}^n)$, we see that there exists a geodesic $\mu_{t,j}^n$, such that, with $A_{t,j}^n:= \spt\mu_{t,j}$, we have 
\[\m^{1/N}(A_{t,j}^n)\geq \sigma_{K/N}^{(1-t)}(\Theta_j^n) \m^{1/N}(A_{0,j}^n)+\sigma_{K/N}^{(t)}(\Theta_j^n)\m^{1/N'}(A_{1,j}^n)\quad \forall t\in [0,1], \]
where $\Theta_j^n:=\Theta(A_{0,j}^n, A_{1,j}^n)$. Define $\mu_t^n:= \sum_{j=1}^{M_n} \lambda_j^n \m_{A_{t,j}^n}$. By the same argument as in \Cref{TBMtoTMCP} using \Cref{midpointapproxlemma}, we see that $\mu_t^n$ converges subsequentially to a $t$-midpoint of $\mu_0$ and $\mu_1$. As the sets $A_{0,j}^n$ are mutually disjoint, from the stable couplings property together with \Cref{singularmidpoint}, the sets $A_{t,j}^n$ are also mutually disjoint. Therefore 
\begin{align*}
\Ent(\mu_t^n) &=\sum_{j=1}^{M_n} \Ent(\lambda_j^n\m_{A_{t,j}^n}) = \sum_{j=1}^{M_n} \lambda_j^n\log\frac{\lambda_j^n}{\m(A_{t,j}^n)}.
\end{align*}
Therefore we have 
\begin{align*}
U_N(\mu_t^n) &=\exp\left(-\frac{\Ent(\mu_t^n)}{N}\right) =\prod_{j=1}^{M_n} \left(\frac{\m(A_{t,j}^n)}{\lambda_j^n}\right)^{\lambda_j^n/N}\\
& \geq \prod_{j=1}^{M_n}\left(\sigma_{K/N}^{(1-t)}(\Theta_j^n)\frac{\m^{1/N}(A_{0,j}^n)}{(\lambda_j^n)^{1/N}}+\sigma_{K/N}^{(t)}(\Theta_j^n)\frac{\m^{1/N}(A_{1,j}^n)}{(\lambda_j^n)^{1/N}}\right)^{\lambda_j^n }
\end{align*}
From this point, one can apply the same argument that was used in \Cref{midpointTCDe} to arrive at
\begin{align*}
U_N(\mu_t^n) &\geq  \sigma_{K/N}^{(1-t)}\left( \sqrt{\sum_{j=1}^{M_n} \lambda_j^n(\Theta_j^n)^2}\right)U_N(\mu_0^n)+\sigma_{K/N}^{(t)}\left(\sqrt{\sum_{j=1}^{M_n} \lambda_j^n(\Theta_j^n)^2}\right)U_N(\mu_1^n)\\
&\geq \sigma_{K/N}^{(1-t)}\left( \Theta^n\right)U_N(\mu_0^n)+\sigma_{K/N}^{(t)}\left(\Theta^n\right)U_N(\mu_1^n).
\end{align*}
Taking an appropriate subsequential limit, one arrives at 
\begin{align*}
U_N(\mu_t)\geq \sigma_{K/N}^{(t)}(\Theta)U_N(\mu_0)+\sigma_{K/N}^{(1-t)}(\Theta)U_N(\mu_1),
\end{align*}
which implies that $X\in \TCD_q^e(K,N)$, by \Cref{midpointTCDe}.
\end{proof}

\subsection{Equivalence between $\sTBM_q(K,N^+)$ and $\TCD_q(K,N)$}
We now reproduce the results of the previous section to prove that $X\in \TCD_q(K,N)$ iff $X\in \sTBM_q(K,N')$ for all $N'>N$. First we need to establish a midpoint lemma for the $\TCD_q$ condition, in analogy with \Cref{midpointTCDe}. 
\begin{lemma}[$\TCD_q(K,N)$ for $t$-midpoints between totally timelike pairs]\label{midpointTCD}
Fix $0<q<1$ and let $(X,\ell,\tau,\m)$ be a measured timelike $q$-essentially non-branching globally hyperbolic synthetic spacetime with $\ell_+$ continuous. Suppose for each totally timelike pair $(\mu_0,\mu_1)\in \Pcal_c^{ac}(X)^2$, and each $t\in[0,1]$, there is a $t$-midpoint $\mu_t$ such that for all $N'>N$, 
\[S_{N'}(\mu_t)\leq \tau_{K,N'}^{(1-t)}(\Theta)S_{N'}(\mu_0)+ \tau_{K,N'}^{(1-t)}(\Theta)S_{N'}(\mu_1).\]
Then $X\in \TCD_q(K,N)$. 
\end{lemma}
\begin{proof}
By recreating the proof of \Cref{midpointTMCP}, we see that the assumption implies that each totally timelike pair $(\mu_0,\mu_1)\in \Pcal_c^{ac}(X)$ admits a geodesic $\mu_t$ such that  
\[S_{N'}(\mu_t)\leq \tau_{K,N'}^{(1-t)}(\Theta)S_{N'}(\mu_0)+ \tau_{K,N'}^{(1-t)}(\Theta)S_{N'}(\mu_1)\quad \forall t\in [0,1].\]
It's sufficient to prove that this implies that $X\in \TCD_q(K,N)$. So let $(\mu_0, \mu_1)\in \Pcal(X)^2$ be $q$-timelike. By the same argument as in \Cref{midpointTCDe}, we obtain sequences $\mu_i^n:=\sum_{k=1}^{M_n}\lambda_k^n \mu_{ik}^n$ for each $i=0,1$, such that $\mu_i^n\narrowto \mu_i$, with corresponding optimal couplings $\pi^n:=\sum_{k=1}^{M_n}\lambda_k^n \pi_k^n$ such that $\pi_k^n \in \TOpt_q(\mu_{0k}^n, \mu_{1k}^n)$, and $\pi^n\in \TOpt_q(\mu_0^n, \mu_1^n)$. Moreover the pairs $(\mu_{0k}^n, \mu_{1k}^n)$ are totally timelike, $\mu_{ik}^n\perp \mu_{ik'}^n$ for all $k\neq k'$, and we have $S_{N'}(\mu_i^n)\to S_{N'}(\mu_i)$ for all $N'\in \N$. Finally we have that 
\begin{align}\label{Thetaellestimate}
|\Theta_k^n-\ell(x,y)|<\frac 1n\quad \forall (x,y)\in \spt\mu_{0k}^n\times \spt\mu_{1k}^n,
\end{align}
where 
\[\Theta_k^n:= \begin{cases}
\inf_{\spt\mu_{0k}^n\times \spt\mu_{1k}^n}\ell(x,y) & K\geq 0;\\
\sup_{\spt\mu_{0k}^n\times \spt\mu_{1k}^n}\ell(x,y) & K<0.
\end{cases}
\]

Since $(\mu_{0k}^n, \mu_{1k}^n)$ are totally timelike, there is a geodesic $\mu_{tk}^n$ such that for all $N'>N$, and all $t\in [0,1]$, 
\[S_{N'}(\mu_{tk}^n)\leq \tau_{K,N'}^{(1-t)}(\Theta_k^n)S_{N'}(\mu_{0k}^n)+ \tau_{K,N'}^{(1-t)}(\Theta_k^n)S_{N'}(\mu_{1k}^n).\]
We define $\mu_t^n:=\sum_{k=1}^{M_n}\lambda_k^n \mu_{tk}^n$, and observe that $\mu_t^n$ converges subsequentially to a geodesic $\mu_t$, by global hyperbolicity. Mutual singularity at the endpoints then gives mutual singularity along the entire geodesic $\mu_{t,k}^n$, and so, for each $t\in[0,1]$ we obtain
\begin{align*}
S_{N'}(\mu_t^n) &= \sum_{k=1}^{M_n}(\lambda_k^n)^{1-1/N'}S_{N'}(\mu_{t,k}^n)\leq \sum_{k=1}^{M_n} (\lambda_k^n)^{1-1/N'}\left( \tau_{K,N'}^{(1-t)}(\Theta_k^n)S_{N'}(\mu_{0k}^n)+ \tau_{K,N'}^{(1-t)}(\Theta_k^n)S_{N'}(\mu_{1k}^n)\right)\\
& \overset{(*)}{\leq} - \sum_{k=1}^{M_n}\int \tau_{K,N'}^{(1-t)}\left(\ell(x,y)\mp \frac1n\right)\left(\frac{\rho_{0k}^n(x)}{\lambda_k^n}\right)^{-1/N'}+ \tau_{K,N'}^{(t)}\left(\ell(x,y)\mp \frac1n\right)\left(\frac{\rho_{1k}^n(y)}{\lambda_k^n}\right)^{-1/N'}d\lambda_k^n\pi_k^n(x,y)\\
&= -\int \tau_{K,N'}^{(1-t)}\left(\ell(x,y)\mp \frac 1n\right)(\rho_0^n)^{-1/N'}+ \tau_{K,N'}^{(t)}\left(\ell(x,y)\mp \frac1n\right)\left(\rho_1^n\right)^{-1/N'}d\pi^n(x,y)\\
&\leq  -\int \tau_{K,N'}^{(1-t)}\left(\ell(x,y)\mp \frac 1n\right)(\rho_0^n)^{-1/N'}\chi_{\{\rho_0\geq C\}}+ \tau_{K,N'}^{(t)}\left(\ell(x,y)\mp \frac1n\right)\left(\rho_1^n\right)^{-1/N'}\chi_{\{\rho_1\geq C\}}d\pi^n(x,y)
\end{align*}
where inequality $(*)$ is obtained from \cref{Thetaellestimate}. Now recall that $\pi^n\to \pi$ setwise, which implies that 
\[\int fd\pi^n\to \int fd\pi\] for all bounded functions $f:X\to \R$. Also recall that 
for each $\epsilon>0$, and sufficiently large $n$, we have 
\[\rho_i-\epsilon \leq \rho_i^n\leq (1+\epsilon)\rho_i,\]
which implies that $(\rho_i^n)^{-1/N} \chi_{\{\rho_i\geq C\}}$ is bounded. It follows that for each $C>0$, 
\[S_{N'}(\mu_t)\leq - \int \tau_{K,N'}^{(1-t)}\left(\ell(x,y)\right)\rho_0^{-1/N'}\chi_{\{\rho_0\geq C\}}+ \tau_{K,N'}^{(t)}\left(\ell(x,y)\right)\rho_1^{-1/N'}\chi_{\{\rho_1\geq C\}}d\pi(x,y).\]
Taking a limit as $C\to 0$ gives the result, by monotone convergence theorem. 
\end{proof}
\begin{theorem}\label{sTBMtoTCD} Fix $0<q<1$. Let $X$ be a timelike $q$-essentially non-branching globally hyperbolic synthetic spacetime with $\ell_+$ continuous. If $X\in \sTBM_q(K,N')$ for all $N'>N$, then $X\in \TCD_q(K,N)$. 
\end{theorem}
\begin{proof}
Let $(\mu_0,\mu_1)\in \Pcal_{c}^{ac}(X)^2$ be a totally timelike pair. By the same proof as in \Cref{sTBMtoTCDe}, and keeping the same notation, we obtain a good simple sequence $(\mu_0^n, \mu_1^n)$ with 
\[\mu_i^n:=\sum_{j=1}^{M_n}\lambda_j^n\m(A_{i,j}^n),\]
where $\diam_d(A_{i,j}^n)<\delta^n$ for some metrizer $d$ and $\delta^n\to 0$, which we are free to specify later. For each $N'>N$, we may then obtain a $t$-midpoint $\mu_t$  between $\mu_0$ and $\mu_1$ as a (subsequential) limit of 
\[\mu_t^n:= \sum_{j=1}^{M_n}\lambda_j^n \m_{A_{t,j}^n}.\]
Each of the $A_{t,j}^n$ are disjoint, and satisfy 
\[\m^{1/N'}(A_{t,j}^n) \geq \tau_{K,N'}^{(1-t)}(\Theta_{j}^n)\m^{1/N'}(A_{0,j}^n)+\tau_{K,N}^{(t)}(\Theta_j^n)\m^{1/N'}(A_{1,j}^n).\] 

Recall that for each $N'>N$, $A_{t,j}^n$ is the support of a geodesic joining $\m_{A_{0,j}^n}$ to $\m_{A_{1,j}^n}$ satisfying the above inequality. But by hypothesis, geodesics are unique, so there is one geodesic satisfying the above inequality for all $N'>N$. Consequently we obtain
\begin{align*}
S_{N'}(\mu_t^n) &= - \int \left(\sum_j^{M_n}\lambda_j^n \frac{\chi_{A_{t,j}^n}}{\m(A_{t,j}^n)}\right)^{1-1/N'}d\m = -\sum_{j=1}^{M_n} (\lambda_j^n)^{1-1/N'}\m^{1/N'}(A_{t,j}^n)\\
&\leq -\sum_{j=1}^{M_n} (\lambda_j^n)^{1-1/N'} \left(\tau_{K,N'}^{(1-t)}(\Theta_{j}^n)\m^{1/N}(A_{0,j}^n)+\tau_{K,N'}^{(t)}(\Theta_j^n)\m^{1/N'}(A_{1,j}^n)\right)\\
&\leq  - \tau_{K,N'}^{(1-t)}(\Theta^n)\sum_{j=1}^{M_n} (\lambda_j^n)^{1-1/N'} \m^{1/N}(A_{0,j}^n)-\tau_{K,N'}^{(t)}(\Theta^n)\sum_{j=1}^{M_n}(\lambda_j^n)^{1-1/N'}\m^{1/N'}(A_{1,j}^n)\\
&= \tau_{K,N'}^{(1-t)}(\Theta^n)S_{N'}(\mu_0^n)+ \tau_{K,N'}^{(1-t)}(\Theta^n)S_{N'}(\mu_1^n).
\end{align*}
Taking an appropriate subsequential limit gives 
\[S_{N'}(\mu_t)\leq \tau_{K,N'}^{(1-t)}(\Theta)S_{N'}(\mu_0)+ \tau_{K,N'}^{(1-t)}(\Theta)S_{N'}(\mu_1),\]
which implies that $X\in \TCD_q(K,N)$, by \Cref{midpointTCD}.  
\end{proof}

\printbibliography
\end{document}